\documentclass[french,11pt]{smfartVScomptage}
\usepackage[latin1]{inputenc}
\usepackage[T1]{fontenc}
\usepackage{lmodern}
\usepackage{smfenum,smfthm,amssymb,mathrsfs,euscript,color}
\usepackage{stmaryrd,mathabx,upgreek,mathdots,mathtools,amscd}
\input{xypic}

\numberwithin{equation}{section} 
\theoremstyle{plain}

\def\AA{\mathbf{A}}
\def\CC{\mathbb{C}}

\def\ZZ{\mathbb{Z}}

\def\B{{\rm B}}
\def\C{{\rm C}}
\def\D{{\rm D}}
\def\E{{\rm E}}
\def\F{{\rm F}}
\def\G{{\rm G}}
\def\H{{\rm H}}

\def\J{{\rm J}}
\def\K{{\rm K}}
\def\L{{\rm L}}
\def\M{{\rm M}}
\def\N{{\rm N}}
\def\P{{\rm P}}
\def\Q{{\rm Q}}
\def\R{{\rm R}}
\def\SS{{\rm S}}
\def\T{{\rm T}}
\def\U{{\rm U}}
\def\V{{\rm V}}
\def\W{{\rm W}}
\def\X{{\rm X}}
\def\Y{{\rm Y}}
\def\Z{{\rm Z}}

\def\Hh{\EuScript{H}}

\def\Oo{\EuScript{O}}

\def\Vv{\mathscr{V}}

\def\a{\alpha} 
\def\b{\beta}
\def\d{\delta}

\def\g{\gamma}
\def\h{\varphi}
\def\k{\kappa}
\def\l{\lambda}

\def\n{\eta}

\def\p{\mathfrak{p}}
\def\s{\sigma}
\def\t{\theta}

\def\w{\varpi}

\def\ie{c'est-à-dire }

\def\>{\geqslant}
\def\<{\leqslant}

\def\Hom{{\rm Hom}}
\def\End{{\rm End}}

\def\Mat{{\rm M}}
\def\GL{{\rm GL}}
\def\Gal{{\rm Gal}}
\def\Ker{{\rm Ker}}

\def\Ind{{\rm Ind}}
\def\ind{{\rm ind}}

\def\qlb{{\overline{\mathbb{Q}}_\ell}}
\def\zlb{{\overline{\mathbb{Z}}_\ell}}
\def\flb{{\overline{\mathbb{F}}_{\ell}}}

\def\ip{\boldsymbol{i}}

\def\r{{\textbf{\textsf{r}}}}
\def\i{\boldsymbol{i}}

\def\kk{\boldsymbol{k}}

\def\ll{\kk_{\C}}

\def\bn{{\eta}} 
\def\bk{{\kappa}} 
\def\bl{{\lambda^{0}}} 
\def\bs{\xi^0}

\def\Ext{{\rm Ext}}
\def\NF{{\sf H}}

\def\Irr{{\rm Irr}} 
\def\Mod{\boldsymbol{{\sf Mod}}}

\def\Rep{\boldsymbol{{\sf Rep}}}
\def\rep{\boldsymbol{{\sf rep}}}
\def\rl{\boldsymbol{{\sf r}}_{\ell}}
\def\tP{\widetilde{\P}}

\def\MM{\boldsymbol{{\sf M}}}

\def\GG{\boldsymbol{{\sf F}}}
\def\GGG{\boldsymbol{{\sf G}}}
\def\GB{{\sf G}}
\def\SB{{\sf S}}
\def\PB{{\sf P}}
\def\TB{{\sf T}}
\def\MB{{\sf M}}
\def\aa{{\sf a}}

\def\nn{{\sf t}}
\def\tt{{\sf t}}
\def\uu{{\sf u}}
\def\vv{a}

\def\mp{\zeta}
\def\xp{\d}
\def\scusp{{\rm scusp}}

\long\def\MSC#1\EndMSC{\def\arg{#1}\ifx\arg\empty\relax\else
     {\par\narrower\noindent%
     2010 Mathematics Subject Classification: #1\par}\fi}

\long\def\KEY#1\EndKEY{\def\arg{#1}\ifx\arg\empty\relax\else
	{\par\narrower\noindent Keywords and Phrases: #1\par}\fi}

\title[Blocs de représentations $\ell$-modulaires de longueur finie]
{Décomposition en blocs de la catégorie des représentations
$\ell$-modulaires lisses de longueur finie de $\GL_m(\D)$} 
\author{Bastien Drevon}
\address{
Laboratoire de Mathémati\-ques de Versailles, 
UVSQ, 
CNRS, 
Université Paris-Saclay,
78035 Versailles, France}
\email{bast139@hotmail.fr}
\author{Vincent Sécherre}
\address{
Laboratoire de Mathémati\-ques de Versailles, 
UVSQ, 
CNRS, 
Université Paris-Saclay,
78035 Versailles, France, Institut Universitaire de France}
\email{vincent.secherre@uvsq.fr}

\begin{abstract}
Soit $\F$ un corps localement compact non archimédien 
de caractéristique résiduelle~$p$,
soit $\G$ une forme intérieure de $\GL_n(\F)$ avec $n\>1$,
et soit $\ell$ un nombre premier différent de $p$.~Nous 
décrivons la décomposition en blocs de la catégorie des 
représentations lisses et de longueur finie de $\G$ à coefficients
dans $\flb$.
Contrairement au cas des représentations complexes d'un 
groupe~réductif $p$-adique quelconque 
et au cas des représentations $\ell$-modulaires de $\GL_n(\F)$,
à chaque bloc de cette décomposition
correspond non pas un unique support supercuspidal,
mais une réunion finie de tels supports,
que nous décrivons. 
Nous prouvons également qu'un bloc supercuspidal est équivalent~au 
bloc principal (\ie le bloc contenant le caractère trivial)
du groupe multiplicatif~d'une~algè\-bre à division convenable,
et nous déterminons les représentations~ir\-réductibles ayant une
extension non scindée avec une représentation supercuspidale de $\G$
donnée. 
\end{abstract}

\begin{altabstract}
Let $\F$ be a non-Archimedean locally compact field of residue characteristic 
$p$, 
let~$\G$ be an inner form of $\GL_n(\F)$ with $n\>1$,
and let $\ell$ be a prime number different from $p$.
We describe the block decomposition of the category of finite length smooth
representations of $\G$ with coefficients in $\flb$.
Unlike the case of complex representations of an arbitrary
$p$-adic reductive group
and that~of $\ell$-modular representations of $\GL_n(\F)$,
several non-isomorphic supercuspidal supports
may correspond to the same block.
We describe the (finitely many)
supercuspidal supports corresponding to a given block.
We also prove that a supercuspidal block is equivalent to the principal
(that is, the one which contains the trivial character)
block of the multiplicative group of a suitable division algebra,
and we determine those irreducible representations having a nontrivial
extension with a given supercuspidal representation of $\G$.
\end{altabstract}

\begin{document} 

\maketitle

\MSC 22E50 
\EndMSC
\KEY 
Bloc, Extension, Longueur, Réduction mod $\ell$,
Représentation supercuspidale, Type
\EndKEY

\thispagestyle{empty}

\section{Introduction}

\subsection{}

Soit $\F$ un corps commutatif localement compact non archimédien de 
caractéristique résiduelle $p$,
et soit $\G$ le groupe des points rationnels d'un groupe algébrique réductif
connexe sur $\F$.
C'est un groupe localement compact et totalement discontinu.
On s'intéresse aux représentations~lis\-ses de $\G$ sur des espaces vectoriels 
complexes et aux opérateurs d'entrelacement entre~ces~repré\-sen\-tations,
qui forment une catégorie abélienne notée $\Rep_\CC(\G)$.
Dans \cite{Bernstein},
Bernstein a montré~com\-ment cette catégorie se décompose en un produit de
blocs, \ie de facteurs directs~in\-dé\-com\-posables,
chacun correspondant bijectivement à une classe inertielle de paires
cuspidales de $\G$.
Au coeur de ce résultat,
il y a le fait qu'une représentation irréductible
cuspidale complexe~de $\G$ est projective modulo le centre,
\ie projective dans la sous-catégorie pleine des repré\-sentations
de $\G$ ayant un caractère central fixé.
La sous-catégorie pleine $\rep_\CC(\G)$ formée des~re\-présentations de
longueur finie 
se décompose elle aussi en une somme directe de blocs,
chacun correspondant cette fois-ci à une \textit{unique} classe de
$\G$-conjugaison de paires cuspidales,
comme~ex\-pliqué par Casselman dans \cite{Casselman} 7.3.

\subsection{}

Remplaçons maintenant le corps $\CC$ des nombres complexes par une clôture
algébrique $\flb$ d'un corps fini de caractéristique un nombre premier $\ell$
différent de $p$.
Les représentations lisses de $\G$ sur des $\flb$-espaces vectoriels,
dites $\ell$-\textit{modulaires},
sont étudiées depuis les travaux fondateurs de Vignéras \cite{Vigb},
dans l'objectif d'étudier les phénomènes de congruence entre formes
automorphes, ainsi que les propriétés de congruence des phénomènes de 
réciprocité et de fonctorialité de~Lang\-lands locales. 
Si la théorie des représentations $\ell$-modulaires des groupes réductifs
$p$-adiques~res\-semble à la théorie complexe sur certains points,
du fait que $p$ est inversible dans $\flb$,
une~dif\-fé\-ren\-ce essentielle est l'existence,
dans le cas modulaire,
de représentations cuspidales non~super\-cuspidales,
\ie apparaissant non comme quotients mais comme sous-quotients d'induites
paraboliques propres
(on renvoie au paragraphe \ref{notip} ci-dessous pour les définitions de
\textit{cuspidal} et \textit{supercuspidal}).
Par conséquent,
une représentation cuspidale n'est en général pas projective~mo\-dulo
le centre dans le cas modulaire.
S'ajoute à ceci le phénomène récemment observé
(\cite{dudassc,datsc}) de non-unicité du support supercuspidal~:
une $\flb$-représentation irréductible d'un groupe réductif $p$-adique
peut apparaître comme sous-quotient d'induites paraboliques
de paires supercuspidales non conjuguées. 
Aussi l'approche de Bernstein devient-elle inopérante dans le cas modulaire. 

\subsection{}
\label{decvsss}

Supposons maintenant que $\G$ soit une forme intérieure du groupe
linéaire $\GL_n(\F)$, avec $n\>1$.
C'est un groupe de la forme $\GL_m(\D)$,
où $\D$ est une algèbre à division centrale de dimension~$d^2$ sur $\F$, 
et où $m$ est un diviseur de $n$ tel que $md=n$. Pour un tel grou\-pe,
on dispose de l'arsenal technique de la théorie des types de
Bushnell et Kutzko développée dans \cite{BK},
\cite{Broussous}, \cite{Grabitz}, \cite{VS1},\dots,
\cite{SeSt2} et adaptée au cas $\ell$-modulaire dans \cite{MSt},
permettant d'étudier en détail ses représenta\-tions $\ell$-modulaires.  
On peut attacher à toute représentation irréductible de $\G$
un unique support~su\-percuspidal
(voir le paragraphe \ref{defscusp}),
et montrer que la catégorie abélienne~$\Rep_\flb(\G)$
des ~re\-pré\-sen\-ta\-tions
lis\-ses de $\G$ à coefficients dans $\flb$
se décompose en un produit de blocs
$\Rep_\flb(\G,\Omega)$, cha\-cun
cor\-res\-pon\-dant bijectivement à une classe d'inertie 
$\Omega$ de pai\-res supercuspidales
(théorème \ref{decbvss}).
Se pose ensuite la question de la décomposition de la
sous-catégorie pleine $\rep_\flb(\G)$,~for\-mée~des repré\-sentations
de longueur finie~: nous y répondons dans le présent article. 

\subsection{}
\label{par14}

Pour décomposer $\rep_\flb(\G)$,
nous nous appuyons sur le résultat suivant (lemme \ref{vanessa}).

\begin{lemm}
Soit $\SS$ une partie de l'ensemble $\X$ des classes d'isomorphisme
de $\flb$-repré\-sen\-tations irréductibles~de $\G$ telle que, 
pour toutes représentations $\tau\in\SS$ et $\pi\in\X-\SS$,
le premier espace d'extension $\Ext^1_\G(\tau,\pi)$ soit nul.
Alors la catégorie
$\rep_\flb(\G)$ se décompose en la somme~di\-rec\-te de $\rep_\flb(\G,\SS)$, 
la sous-catégorie pleine des re\-pré\-sentations dont les
sous-quotients~ir\-ré\-duc\-tibles sont dans $\SS$,
et de $\rep_\flb(\G, \X-\SS)$.
\end{lemm}

Pour que $\Ext^1_\G(\tau,\pi)$ soit non nul,
il faut et suffit qu'il existe une représentation indécompo\-sa\-ble
de $\G$ de longueur $2$ dont l'unique sous-représentation irréductible soit
$\pi$ et l'unique quotient irréductible soit $\tau$.
\'Etant donné une représentation~irré\-duc\-ti\-ble $\pi$ de $\G$,
nous cherchons donc à quelle condition une représentation 
$\tau\in\X$ a une extension non scindée avec $\pi$, 
et plus généralement à quelle condition $\pi$ et $\tau$ sont des constituants
irréductibles d'une représentation~indé\-com\-po\-sa\-ble
de $\rep_\flb(\G)$.
Pour cela, 
la décomposition en blocs de $\Rep_\flb(\G)$ donnée
au~paragra\-phe~\ref{decvsss}
assure qu'on peut~se~ramener au cas où $\pi$ et $\tau$ ont des supports
supercuspidaux iner\-tiel\-le\-ment équivalents.

\subsection{}
\label{defequi0}

Notre stratégie repose partiellement sur la notion de réduction modulo $\ell$. 
Expliquons de~quoi il s'agit.
Considérons les représentations~irré\-ductibles de $\G$ à coefficients dans
une clôture~algé\-bri\-que $\qlb$ du corps des nombres~$\ell$-adi\-ques.~Une
telle représentation est dite \textit{entière} si elle admet
un $\zlb$-réseau stable par $\G$,
où $\zlb$ est~l'anneau des entiers de $\qlb$.
Tensoriser un tel réseau par~$\flb$ fournit une re\-pré\-sentation
lisse $\ell$-modulaire~de longueur finie,
dont la semi-simplification ne~dé\-pend pas du réseau choisi~:
on appelle celle-ci~la \textit{réduction mod $\ell$} de la
$\qlb$-représentation entière considérée.
Par exemple,
il y a un critère~sim\-ple pour savoir si une
$\qlb$-représentation cuspidale $\mu$ de $\G$
est entière~:
il faut et il suf\-fit que son ca\-ractère central soit lui-même entier,
\ie à valeurs dans $\overline{\ZZ}{}^\times_\ell$.
Si tel est le cas,~il
y a une $\flb$-représentation irréductible cuspidale $\pi$ de $\G$
telle que la réduction mod $\ell$ de $\mu$ soit~:
\begin{equation}
\label{labelrhointro}
\pi \oplus \pi\nu \oplus \dots \oplus \pi\nu^{a-1}
\end{equation}
où $\nu$ désigne le caractère ``valeur absolue de la norme réduite'' de $\G$ 
et $a=a(\mu)\>1$ la longueur de~cette réduction mod $\ell$.
En outre,
l'ensemble des facteurs irréductibles de \eqref{labelrhointro} est soit réduit 
à $\pi$,
soit formé de tous les $\pi\nu^j$, $j\in\ZZ$,
et on se trouve dans l'un ou l'autre cas selon que $\ell$ divise ou non
l'entier $q(\pi)-1$,
où $q(\pi)$ est une certaine puissance du cardinal $q$ du corps résiduel de 
$\F$ associée à $\pi$ au paragraphe \ref{flapflip}.
Dans le cas où $\pi$ est de niveau $0$,
elle vaut simplement $q^n$.

Inversement,
si $\pi$ est une $\flb$-représentation \textit{supercuspidale} 
(et pas seulement cuspidale) de $\G$,
l'ensemble des représentations de $\G$
apparaissant avec $\pi$
dans la réduction mod $\ell$ d'une même $\qlb$-re\-pré\-sentation 
cuspidale entière est soit réduit à $\pi$ (si $\ell$ ne divise pas $q(\pi)-1$),
soit formé des représentations $\pi\nu^j$, $j\in\ZZ$
(si $\ell$ divise $q(\pi)-1$).

Observons que,
si $\kk$ est un corps fini de caractéristique $p$,
il existe un phénomène comparable mais plus simple 
dans le cas du groupe linéaire général $\GL_m(\kk)$~: 
une $\qlb$-représentation~irréduc\-tible de $\GL_m(\kk)$
dont la réduction mod $\ell$ contient une $\flb$-représentation
supercuspidale
est~elle-même
cuspidale et sa réduction mod $\ell$ est irréductible.

\subsection{}
\label{defequi}

Disons plus généralement
que des représentations irréductibles $\ell$-modulaires de $\G$
sont~\textit{équi\-valentes} s'il existe une $\qlb$-représentation irréductible
entière de $\G$ dont la réduction mod $\ell$ con\-tienne chacune de
ces~re\-présentations.
Il n'est pas évident \textit{a priori} qu'il s'agisse d'une relation 
d'équi\-valence,
mais on peut montrer le résultat suivant (proposition \ref{blocindec}).

\begin{prop}
\label{lapintro}
Soit $\pi$ une représentation irréduc\-tible de $\G$
dont le support supercuspidal est $\rho_1+\dots+\rho_r$.
Alors une représentation irréductible $\pi'$ est équivalente à $\pi$ si et seulement
s'il y a des entiers $j_1,\dots,j_r\in\ZZ$ tels que le support 
supercuspidal de $\pi'$ soit
$\rho_1\nu^{j_1}+\dots+\rho_r\nu^{j_r}$,
où, pour chaque $k=1,\dots,r$, l'entier $j_k$ est nul 
si $\ell$ ne divise pas $q(\rho_k)-1$.
\end{prop}

Notons $\B(\pi)$ la classe des représentations irréductibles
équivalentes à une représentation donnée $\pi$, 
et notons $\mathscr{B}$ l'ensemble des classes $\B(\pi)$
lorsque $\pi$ décrit les~re\-pré\-sentations irréductibles 
$\ell$-modu\-laires de $\G$.
(Attention~:
la définition des $\B(\pi)$ que nous donnons au paragraphe \ref{defbepi}
est différente, 
mais équivalente d'après la proposition \ref{lapintro}.)
On a le résultat suivant (théorème \ref{Prop 3.2.4}).

\begin{theo} 
On a une décomposition en blocs~:
\begin{equation*}
\label{decbintro}
\rep_\flb (\G) = \bigoplus\limits_{\B} \rep_\flb (\G,\B)
\end{equation*}
où $\B$ décrit les éléments de $\mathscr{B}$,
et où $\rep_\flb (\G,\B)$ est la sous-catégorie pleine
de $\rep_{\flb}(\G)$ formée des représentations dont tous
les sous-quotients irréductibles sont dans $\B$.  
\end{theo}

Dans le cas où $\G=\GL_n(\F)$, 
des représentations irréductibles de $\G$
sont dans le même bloc si et seulement si elles ont le
même support supercuspidal
(voir les remarques \ref{Bsingleton} et \ref{Bsingleton6}).

\subsection{}

Il n'est pas difficile de montrer
(voir le lemme \ref{lemdep})
que deux représentations irréductibles~équi\-valentes au sens du paragraphe 
\ref{defequi} sont les constituants d'une représentation indécompo\-sa\-ble~de
longueur finie $\G$,
et (voir le lemme \ref{lemmecc})
que les constituants irréductibles
d'une représentation~in\-dé\-composa\-ble~de 
longueur finie $\G$ ont le même caractère central. 
Il ne reste donc qu'à prouver que,
pour toute classe $\B\in\mathscr{B}$, 
l'espace $\Ext^1_\G(\tau,\pi)$ est nul quels que soient
$\pi\in\B$ et $\tau\in\X-\B$.
Nous pouvons supposer,
comme observé au paragraphe \ref{par14},
que $\tau$ et $\pi$ ont des supports~super\-cuspidaux iner\-tiellement
équivalents.  
Nous pro\-cé\-dons~en trois étapes~: 
\begin{enumerate}
\item 
d'abord le cas où $\pi$ est une
représentation cuspidale de niveau $0$ de $\G$, 
\item
puis le cas où $\pi$ est une
représentation cuspidale de niveau quelconque de $\G$,
\item
et enfin le cas général,
\end{enumerate}
que nous traitons respectivement dans les paragraphes \ref{sec4},
\ref{sec5} et \ref{sec6}.
Détaillons chacune de ces trois étapes,
à commencer par la première.

\subsection{}

La première étape repose sur la description de $\pi$ comme l'induite compacte 
d'une représenta\-tion $\xi$ d'un~sous-groupe~$\N$ de $\G$ telle~que~:
\begin{itemize}
\item 
le groupe $\N$ contient et normalise le~sous-groupe compact maximal 
$\GL_m(\Oo_\D)$, où $\Oo_\D$ est l'an\-neau des entiers de $\D$,
\item
la restriction de $\xi$ à $\GL_m(\Oo_\D)$ est l'inflation d'une représentation 
irréducti\-ble~cuspidale $\s$ du groupe fini $\GL_m(\kk_\D)$, 
où $\kk_\D$ est le corps résiduel de $\D$.
\end{itemize}
(Dans le langage de la théorie des types simples de 
Bushnell-Kutzko~\cite{BK,MSt}, 
la paire $(\N,\xi)$ est un type simple maximal étendu de ni\-veau $0$~: 
voir le paragraphe \ref{consnxi3}.)
D'abord
(voir les~paragra\-phes~\ref{consnxi} à \ref{floupifloup}),
nous montrons que la détermination des représentations cuspidales de $\G$
ayant une extension non scindée avec $\pi$ équivaut à la détermination des
caractères non ramifiés $\chi$ de~$\G$ tels que $\xi$ ait une extension
non scindée~avec un conjugué de $\xi\chi$ sous le normalisateur~de~$\N$~dans
$\G$.
Nous ramenons ainsi le problème~ini\-tial à un problème dans la catégorie
$\rep_{\flb}(\N)$ des $\flb$-re\-présentations de~longueur~finie~de
$\N$.

Le groupe $\N$ se comportant presque comme un groupe fini,
nous prouvons ensuite
que $\xi$~ad\-met une enveloppe projective dans $\rep_{\flb}(\N)$,
ce qui permet de ramener~le problème à celui de la description
des $\qlb$-représentations irréductibles $\d$ de $\N$ dont la réduction
mod $\ell$ contient $\xi$ (voir les paragraphes \ref{subsection 3.1.2} à
\ref{par46}).

Pour décrire ces représentations $\d$,
la situation est particulièrement favorable lorsque $\pi$
(ou de façon équivalente $\s$)
est supercuspidale,
grâce au fait que toute $\qlb$-représentation~irréduc\-tible de $\GL_m(\kk_\D)$
dont la réduction mod $\ell$ contient $\s$
est~elle-même cuspidale et sa réduction mod $\ell$ est irréductible~;~les
représentations $\d$ se décrivent alors aisément à 
l'aide des types simples~ma\-xi\-maux de niveau $0$ relevant $\s$.
Si $\pi$
n'est pas supercuspidale,
le fait précédent est à remplacer par~le~fait~que 
toute $\qlb$-re\-pré\-sentation~irréduc\-tible de $\GL_m(\kk_\D)$
dont la réduction mod $\ell$ con\-tient~$\s$
est générique (voir les paragraphes \ref{merlinlenchanteur} à \ref{utherpendragon}).

\subsection{}
\label{par18}

La seconde étape consiste à se ramener au cas précédent,
{au moyen d'une équivalence de~ca\-tégories 
construite par Chinello \cite{gianmarco2,gianmarco3} 
grâce à la théorie des types simples de 
Bushnell-Kutzko~\cite{BK,MSt}.
\'Etant donné une représentation cuspidale~$\pi$ de $\G$,
correspondant à une classe~iner\-tiel\-le $\Omega$, Chinello~: 
\begin{itemize}
\item 
construit un progénérateur de type fini de la catégorie~:
\begin{equation}
\label{endoblocs}
\bigoplus\limits_{\Omega'}\Rep_\flb(\G,\Omega')
\end{equation}
(où $\Omega'$ décrit les classes inertielles équivalentes à $\Omega$
en un sens que nous ne préciserons pas ici), 
\item
et prouve l'existence d'une extension finie $\E$ de $\F$ de degré $k$ divisant 
$n$ et d'une équivalence de catégories~:
\begin{equation*}
\bigoplus\limits_{\Omega'}\Rep_\flb(\G,\Omega')
\smash{\mathop{\longrightarrow}\limits^{\GG}}  
\bigoplus\limits_{\Omega_0}\Rep_\flb(\G_0,\Omega_0)
\end{equation*}
où $\Omega_0$ décrit les clas\-ses inertielles de niveau $0$
d'une certaine forme intérieure $\G_0$ de $\GL_{n/k}(\E)$. 
\end{itemize}
Le foncteur $\GG$ envoie donc 
$\Rep_\flb(\G,\Omega)$ sur $\Rep_\flb(\G_0,\Omega_0)$
pour une certaine classe inertielle~$\Omega_0$.

Il s'agit alors de prouver que $\GG$ envoie $\pi$ sur une
repré\-sentation cuspidale $\pi_0$
(lemme \ref{lemmeheronseiche}),
puis que l'image de~$\B(\pi)$~par $\GG$
est égale à $\B(\pi_0)$,
ce que nous faisons en décrivant le comportement de ce foncteur 
par torsion par un caractère non ramifié
(lemme \ref{adieulescons}).

\subsection{}

Enfin,
nous traitons le cas général dans la section \ref{sec6},
en nous inspirant de \cite{EH} Theorem 3.2.13.
Nous obtenons le résultat suivant
(proposition \ref{Prop 3.2.1}).

\begin{prop}
\label{Prop321intro}
Soient $\pi$ et $\pi'$ des $\flb$-représentations irréductibles de $\G$
telles que l'espace $\Ext^1_\G(\pi',\pi)$ soit non nul.
Alors $\pi$ et $\pi'$ sont équivalentes. 
\end{prop}

Enfin, le résultat suivant (proposition \ref{blocindecfinal})
complète la description de la relation
d'équivalence faite dans la proposition \ref{lapintro}.

\begin{prop}
\label{blocindecfinalintro}
Pour que des représentations~ir\-réductibles de $\G$
soient équivalentes,
il faut et suffit qu'elles apparaissent comme sous-quotients
d'une même représentation indécomposable de longueur finie de $\G$.
\end{prop}

\subsection{}

Dans la section \ref{sec7},
nous nous intéressons aux blocs supercuspidaux
de $\Rep_\flb(\G)$ et $\rep_\flb(\G)$,
dans l'objectif de prolonger les résultats de Chinello 
\cite{gianmarco2,gianmarco3} (voir le paragraphe \ref{par18}). 
Soit $\pi$~une $\flb$-représentation supercuspidale de $\G$,
de classe inertielle $\Omega$,
et posons $\B=\B(\pi)$.
Nous cons\-trui\-sons~un progénérateur de type fini du bloc 
$\Rep_{\flb}(\G,\Omega)$
(proposition \ref{propetitprog})
et calculons l'algèbre~de ses en\-do\-mor\-phismes
(théorème \ref{structureE}).
Nous en déduisons le résultat suivant
(théorème \ref{nerodisepia} et~co\-rol\-laire \ref{corodisepia}).
Appelons \textit{bloc principal} le bloc contenant le caractère trivial.

\begin{theo}
\label{nerodisepiaintro}
Il existe un corps localement compact non archimédien $\F'$
et une $\F'$-algèbre~à division centrale $\D'$ tels que
les blocs $\Rep_{\flb}(\G,\Omega)$et $\rep_{\flb}(\G,\B)$
soient respectivement équiva\-lents aux blocs prin\-ci\-paux
de $\Rep_{\flb}(\D'^\times)$ et $\rep_{\flb}(\D'^\times)$. 
\end{theo}

{Ce résultat corrobore le principe
selon lequel, 
étant donné un groupe réductif connexe ${\bf G}$~dé\-fini~sur $\F$, 
un bloc de $\Rep_{\flb}({\bf G}(\F))$
devrait être équivalent au bloc princi\-pal
de $\Rep_{\flb}({\bf G}'(\F))$ pour un groupe ré\-duc\-tif 
convenable ${\bf G}'$ (voir Dat \cite{Datfonc},
ainsi que \cite{Dattame,gianmarco3} pour le cas de $\GL_n(\F)$).}
Observons cependant que,
contrairement à ce qui se passe pour les représentations complexes,~on
ne peut pas toujours choisir $\D'=\F'$ dans ce théorème
(voir la remarque \ref{paranormal}).

\subsection{}

Dans la dernière section,
étant donné une $\flb$-représentation supercuspidale $\pi$ de $\G$, 
nous~dé\-terminons toutes les~re\-présentations irréductibles $\pi'$ de
$\G$ telles qu'il existe une extension non scindée de $\pi$ par $\pi'$.
Le résultat (proposition \ref{pih})
s'exprime en fonction de l'invariant de Hasse~$h$ de $\D$
(définition \ref{hasse})
et du degré $k$ de l'extension $\E/\F$ du paragraphe \ref{par18}.
On note $(a,b)$ le plus grand diviseur commun à deux entiers $a,b\>1$.

\begin{prop}
\label{pihintro}
Notons $h(\pi)$ le reste dans la division euclidienne de 
$hk/(k,d)$ par $d/(k,d)$.
L'ensemble des re\-pré\-sentations $\pi'$ de $\G$ telles
que $\Ext^1_{\G}(\pi,\pi')$ soit non nul est~:
\begin{enumerate}
\item 
réduit à $\pi$ si $\ell$ ne divise pas $q(\pi)-1$, 
\item
formé de $\pi$ et de la représentation 
$\pi\nu^{-h(\pi)}$ si $\ell$ divise $q(\pi)-1$.
\end{enumerate}
\end{prop}


\subsection{}

Signalons pour finir
qu'un rapporteur d'une version antérieure de cet article a attiré
notre~at\-tention sur le fait que la récurrence sur laquelle est 
fondée la preuve de \cite{SEns} Proposition 1.3~est incorrecte. 
Nous n'avons pas su corriger la récurrence,
mais nous expliquons dans un appendice au présent article 
pourquoi le résultat principal de 
\cite{SEns} --- c'est le théorème \ref{decbvss} ci-dessous ---
est toujours vala\-ble.~Nous remercions le rapporteur 
d'avoir attiré notre attention sur ce point. 

\section{Notations}

Nous introduisons maintenant les principales définitions et notations
utilisées dans la suite. 

\subsection{}

Fixons un corps localement compact non archimédien $\F$,
de caractéristique résiduelle $p$. 

Si $\K$ est une extension finie de $\F$, 
ou plus généralement une $\F$-algèbre à division de dimension finie,
on note $\Oo_\K$ son anneau d'entiers,
$\p_\K$ l'idéal maximal de $\Oo_\K$ et $\kk_\K$ son corps résiduel,
qui est un corps fini de cardinal noté $q_\K$.
On pose $q=q_\F$ dans toute la suite.

Si $n$ est un entier strictement positif,
on note $\Mat_n(\K)$ l'algèbre des ma\-trices carrées de
taille~$n$ à~coefficients dans $\K$
et $\GL_n(\K)$ le groupe de ses éléments inversibles.
Muni de la topologie~in\-duite par celle de $\K$, 
celui-ci est un groupe topologique localement profini. 

\'Etant donné un nombre premier $\ell$ différent de $p$,
on note $\qlb$ une clôture algébrique du corps des nombres $\ell$-adiques,
$\zlb$ son anneau des entiers et $\flb$ son corps résiduel. 

\subsection{}

Soit $\R$ un anneau commutatif,
et soit $\G$ un groupe localement profini.

Par $\R$-\textit{représentation} (ou
simplement représentation si aucune confusion n'en résulte)
de $\G$
on entendra toujours une re\-pré\-sentation lisse sur un $\R$-module. 
Par $\R$-\textit{caractère} de $\G$, on entendra une représentation lisse
de $\G$ sur $\R$, \ie un homomorphisme de groupes de $\G$
dans $\R^\times$ de noyau ouvert. 

On note $\Rep_\R(\G)$ la catégorie abélienne des $\R$-représentations lisses de 
$\G$ et $\rep_\R(\G)$ la~sous-catégorie pleine formée des 
représentations de longueur finie.

Si $\pi$ est une représentation d'un sous-groupe fermé $\H$ de $\G$
et $g\in\G$, on pose $\H^g=g^{-1}\H g$ et on note 
$\pi^g$ la représentation $x\mapsto\pi(gxg^{-1})$ de $\H^g$. 
Si $\chi$ est un caractère de $\H$, 
on note $\pi\chi$ la re\-présentation
$x\mapsto\chi(x)\pi(x)$ de $\H$.

\section{Préliminaires et premiers théorèmes de décomposition}
\label{sec3}

Fixons une fois pour toutes une $\F$-algèbre à division centrale $\D$, 
de degré réduit $d$, et un entier $m\>1$.
Posons $\G=\GL_m(\D)$.
Si l'on pose $n=md$,
c'est une forme intérieure de $\GL_n(\F)$.

On fixe une uniformisante $\w_\F$ de $\F$,
et une uniformisante $\w_\D$ de $\D$ telle que $\w_\D^d=\w_\F^{\phantom{d}}$. 

Dans cette section, $\R$ est un corps algébriquement clos de
caractéristique dif\-férente de $p$.

\subsection{}
\label{notip}

Soit $\P$ un sous-groupe parabolique de $\G$,
soit $\N$ son radical unipotent et soit $\M$ une composante de Levi de $\P$.
On note $\ip_\P^\G$ le foncteur d'induction parabolique normalisée de $\M$
à $\G$ relativement à $\P$.
(La normalisation nécessite de fixer une racine carrée de $q$ dans $\R$,
ce que nous faisons une fois pour toutes.)
Si $m_1,\dots,m_r$ sont des entiers $\>1$ de somme $m$,
si $\M$ est le sous-groupe de Levi des matri\-ces diagonales par blocs de
tailles respectives $m_1,\dots,m_r$, canoniquement isomorphe au produit 
$\GL_{m_1}(\D)\times\dots\times\GL_{m_r}(\D)$,
si $\P$ est le sous-groupe parabolique standard engendré par $\M$ et les
matrices unipotentes triangulaires supérieures de $\G$,
et si $\pi_i$ est une représentation de $\GL_{m_i}(\D)$ pour chaque
$i\in\{1,\dots,r\}$,
on note~:
\begin{equation}
\label{notprod}
\pi_1\times\dots\times\pi_r 
\end{equation}
l'induite parabolique de $\pi_1\otimes\dots\otimes\pi_r$ à $\G$
relativement à $\P$.

Une représentation irréductible de $\G$ est dite \textit{cuspidale}
(resp.\! \textit{supercuspidale})
si elle n'est~quo\-tient
(resp.\! sous-quotient)
d'aucune représentation de la forme \eqref{notprod} avec $r\>2$.
Une représenta\-tion supercuspidale est donc cuspidale,
la ré\-ci\-proque n'étant pas vraie en général.
Si $\R$ est de ca\-ractéristique nulle,
toute représentation irréductible cuspidale est supercuspidale.

On définit de façon analogue les notions de représentation irréductible
cuspidale et supercus\-pi\-dale de $\GL_m(\kk)$,
où $\kk$ est un corps fini de caractéristique $p$.

\subsection{}
\label{defscusp}

Soit $\pi$ une représentation irréductible de $\G$.
Il y a des~entiers $m_1,\dots,m_r$ de somme $m$~et~des 
représentations irréductibles supercuspidales
$\pi_1,\dots,\pi_r$, 
comme au paragraphe \ref{notip}, telles que 
$\pi$ soit un sous-quotient de l'induite parabolique
\eqref{notprod}.
D'après \cite{MSc} Theorème~8.16, 
ces représen\-tations supercuspidales sont uniques à permutation près.  
La somme formelle~:
\begin{equation}
\label{defscusp}
\pi_1+\dots+\pi_r
\end{equation}
s'appelle le \textit{support supercuspidal} de $\pi$.
On la note $\scusp(\pi)$.

On définit et on note de façon analogue le support supercuspidal
d'une re\-pré\-sentation irréduc\-tible de $\GL_m(\kk)$,
dont l'unicité provient par exemple de \cite{MSf} Théorème 2.5.

La \textit{classe inertielle} d'un support supercuspidal \eqref{defscusp}
est l'ensemble des supports supercuspidaux de $\G$ de la forme
$\pi_1\chi_1+\dots+\pi_r\chi_r$ où $\chi_i$ est un caractère non ramifié
du groupe $\GL_{m_i}(\D)$ pour chaque $i\in\{1,\dots,r\}$.
Si $\Omega$ est une classe inertielle de supports supercuspidaux de $\G$,
on note $\Rep_\R(\G,\Omega)$ la sous-ca\-té\-go\-rie pleine de
$\Rep_\R(\G)$ formée des représentations dont tous les sous-quotients
irréductibles ont leur support supercuspidal dans $\Omega$.

Appelons \textit{bloc} de $\Rep_\R(\G)$ un facteur direct indécomposable
de cette catégorie.

\begin{theo}[\cite{Bernstein,Vigs,SEns}]
\label{decbvss}
Pour chaque classe inertielle $\Omega$ de supports supercuspidaux~de $\G$, 
la sous-ca\-té\-go\-rie $\Rep_\R(\G,\Omega)$ est un bloc.
On a une décomposition~:
\begin{equation}
\label{decomega}
\Rep_\R(\G) = \prod\limits_{\Omega} \Rep_\R(\G,\Omega)
\end{equation}
où $\Omega$ décrit les classes inertielles de supports supercuspidaux de 
$\G$. 
\end{theo}

\subsection{}

Passons maintenant à la sous-catégorie $\rep_\R(\G)$ des $\R$-représentations
de longueur finie de~$\G$.
D'abord, la décomposition \eqref{decomega} induit une décomposition~:
\begin{equation*}
\rep_\R(\G) = \bigoplus\limits_{\Omega} \rep_\R(\G,\Omega)
\end{equation*}
où $\rep_\R(\G,\Omega)$ est la sous-ca\-té\-go\-rie pleine des représentations
de longueur finie de $\Rep_\R(\G,\Omega)$. 
Cependant, contrairement aux $\Rep_\R(\G,\Omega)$,
les facteurs $\rep_\R(\G,\Omega)$ ne sont pas indécomposables,
comme nous allons le voir tout de suite.

Notons $\Z$ le centre de $\G$,
naturellement isomorphe à $\F^\times$.
D'après \cite{Vigb} II.2.8,
toute représen\-ta\-tion irréductible de $\G$
admet un caractère central.

\begin{lemm}
\label{lemmecc}
Soit $\V$ une représentation de longueur finie de $\G$. 
Pour tout caractère $\a$~de~$\Z$, 
notons $\V(\a)$ la plus grande sous-représentation de $\V$
dont les sous-quotients~ir\-ré\-ductibles admet\-tent $\a$ pour 
caractère central.
On a alors~:
\begin{equation*}
\label{deccc}
\V = \bigoplus\limits_{\a} \V(\a).
\end{equation*}
\end{lemm}

\begin{proof}
Soit $n$ la longueur de $\V$, 
et soit $r$ le cardinal de l'ensemble des caractères~centraux des
composants irréductibles de $\V$.
On a $n\>r$, et on peut supposer que $r\>2$.
La preuve se fait par récurrence sur $n$,
comme celle~de \cite{SEns} Proposition 1.7.
Nous ne détaillons que le cas où $n=r=2$. 

Supposons donc que $\V$ ait une sous-représentation
irréductible $\W$, de caractère central $\a$,
et que $\V/\W$ soit irréductible et de caractère central $\b\neq\a$.
Il s'agit de prouver que $\W$ a~un~supplé\-mentaire dans $\V$
stable par $\G$.
Plus précisément, nous allons prouver que~:
\begin{equation*}
\X = \{v\in\V\ |\ z\cdot v = \b(z)v \text{ pour tout $z\in\Z$}\}
\end{equation*}
est un supplé\-mentaire de $\W$ dans $\V$ stable par $\G$.
Bien sûr,
$\X$ est stable par $\G$ et $\W\oplus\X\subseteq\V$.
Comme $\V$ est de longueur $2$,
il suffit de prouver que $\X\neq\{0\}$ pour en déduire que
$\W\oplus\X=\V$.  
Fixons un $v\in\V$ tel que $v\notin\W$
et un $z_0\in\Z$ tel que $\b(z_0)\neq\a(z_0)$,
et posons~:
\begin{equation*}
x=z_0\cdot v-\a(z_0)v.
\end{equation*}
Pour tout $z\in\Z$,
on pose $w=z\cdot v-\b(z)v$
(qui appartient à $\W$
car $\Z$~agit~par $\b$ sur le quotient $\V/\W$).
En particulier,
pour $z=z_0$,
on en déduit que $x\neq0$ car $v\notin\W$.
On a~:
\begin{eqnarray*}
z\cdot x 
&=& z_0 \cdot (\b(z)v+w) -\a(z_0)(z\cdot v) \\
&=& \b(z) (x+\a(z_0)v) + \a(z_0)w -\a(z_0)(w+\b(z)v) \\
&=& \b(z)x
\end{eqnarray*}
\ie que $x\in\X$.  

Pour finir, 
indiquons brièvement comment prouver le lemme par récurrence sur $n$.
Si $n=2$, le lemme est prouvé.
Si $n\>3$,
on fixe un caractère central $\a_0$ d'une sous-représentation irréduc\-ti\-ble de 
$\V$, 
et on applique l'hypothèse de récurrence à $\W=\V/\V(\a_0)$,
qui se décompose sous la forme
$\W(\a_1)\oplus\dots\oplus\W(\a_s)$, 
les $\a_i$ étant des caractères de $\Z$ tous distincts de $\a_0$ par
maximalité de $\V(\a_0)$ (on a donc $s=r-1$).
Il ne reste plus qu'à prouver que, pour tout $i\in\{1,\dots,r-1\}$,
il y a une sous-représentation $\V_i$ de $\V$ isomorphe à $\W(\a_i)$~:
on en déduira que $\V$ est la somme~di\-recte
$\V(\a_0)\oplus\V_1\oplus\dots\oplus\V_{r-1}$
et que $\V_i$ est égal à $\V(\a_i)$.
Pour cela,
on raisonne comme dans~la preu\-ve~de \cite{SEns} Lemma 1.8,
en appliquant l'hypothèse de récurrence à la préimage de $\W(\a_i)$ dans $\V$
et en s'aidant du cas où $n=r=2$.
Pour plus de détails, voir \cite{SEns} Proposition 1.7.
\end{proof}

Il sera commode d'introduire la définition suivante.

\begin{defi}
\label{defdep}
Des représentations irréductibles $\pi$, $\pi'$
de $\G$ sont dites \textit{dépendantes}~s'il y a une représentation
indécomposable 
de longueur finie de $\G$ ayant $\pi$ et $\pi'$ pour sous-quotients.
\end{defi}

D'après le lemme \ref{lemmecc},
pour que deux représentations irréductibles de $\G$
soient dépendantes,~il faut qu'elles aient le même caractère central. 

\subsection{}

D'après \cite{Vigsheaves},
la catégorie $\Rep_\R(\G)$ a assez d'objets projectifs. 
On peut donc définir,
pour~des re\-présentations $\pi,\pi'$ de cette catégorie, 
des espaces d'extension $\Ext^i_\G(\pi,\pi')$ pour tout $i\>0$.

Dans ce paragraphe,
nous nous concentrerons sur le premier espace d'extension
entre représentations irréductibles.
\'Etant donné des représentations irréductibles $\pi$ et $\pi'$ de $\G$, 
l'espace~d'ex\-ten\-sion
$\Ext^1_\G(\pi,\pi')$ est non nul si et seulement s'il existe une
représentation indécomposable~de~$\G$ de longueur $2$
dont l'unique sous-représentation soit isomorphe à $\pi'$
et l'unique~quo\-tient soit~iso\-morphe à $\pi$.

Le lemme suivant nous donne un moyen général de décomposer la catégorie 
$\rep_\R(\G)$ en~somme de deux facteurs directs.
Notons $\Irr(\G)$ l'ensemble des classes d'isomorphisme de
repré\-sen\-ta\-tions irréductibles de $\G$.

\begin{lemm}
\label{vanessa}
Soit $\SS$ une partie de $\Irr(\G)$ telle que, 
pour toute $\s\in\SS$ et toute $\pi\in \Irr(\G)-\SS$,
on ait~:
\begin{equation}
\label{deboite}
\Ext^1_\G(\s,\pi) = \{0\}.
\end{equation}
Alors $\rep_\R(\G)$ se décompose en la somme directe de $\rep_\R(\G,\SS)$, 
la sous-catégorie pleine des~re\-pré\-sentations dont tous les sous-quotients
irréductibles sont dans $\SS$,
et de $\rep_\R(\G, \Irr(\G)-\SS)$.
\end{lemm}

\begin{proof}
Soit $\V$ une représentation de longueur finie de $\G$.
On note $\X$ la plus~gran\-de sous-représentation de $\V$ dont tous les 
sous-quotients irréductibles sont dans $\SS$,
et $\Y$ celle~dont tous les sous-quotients irréductibles sont hors de $\SS$.
Il s'agit de prouver que $\V$ est égale à $\X\oplus\Y$,
\ie que $\W=\V/(\X\oplus\Y)$ est nul.
Supposons que ce ne soit pas le cas~;
soit $\pi$ une~sous-représentation irréductible de $\W$,
qu'on peut supposer dans $\SS$ pour fixer les idées,
le cas contraire se traitant de la même façon.
Notons $\U$ l'image réciproque de $\pi$ par la surjection naturelle de~$\V$ 
sur $\W$.
C'est une extension de $\U/\Y$,
qui a tous ses sous-quotients irréductibles dans $\SS$,
par $\Y$,
qui a tous ses sous-quotients irréductibles hors de $\SS$.
Par un argument de dévissage classique,~la condition
\eqref{deboite} entraîne que
$\Ext^1_\G(\U/\Y,\Y)$ est nul~;
ainsi $\U$ est la somme directe de $\Y$ et $\U/\Y$. Le
fait que $\U/\Y$ ait tous ses sous-quotients irréductibles dans $\SS$,
contienne~strictement~$\X$~et~se plon\-ge~dans $\V$
con\-tredit la maximalité de $\X$.
\end{proof}

\subsection{}

Le lemme suivant donne
une condition suffisante pour que des $\flb$-représentations
irréductibles soient dépendantes. 
Elle repose sur la notion de réduction mod $\ell$,
que nous rappelons maintenant.

Soit $\pi$ une représentation de longueur finie de $\G$ sur un
$\qlb$-espace vectoriel $\V$.
Elle est dite~\textit{en\-tière} si $\V$ contient un $\zlb$-réseau
(\ie un sous-$\zlb$-module libre engendré par une base de $\V$
sur $\qlb$)
stable par $\G$.
Si $\L$ est un tel réseau,
la représentation de $\G$ sur $\L\otimes\flb$ est alors lisse et de longueur
finie, et sa semi-simplification ne dépend que de $\pi$,
et pas du choix de $\L$. 
On la note $\rl(\pi)$, qu'on appelle la \textit{réduction mod $\ell$} de 
$\pi$. 
Pour les détails, on~ren\-voie le lecteur à \cite{Vigb,Vigw}.

\begin{lemm}
\label{lemdep}
Soit $\pi$ une $\qlb$-représentation irréductible entière de $\G$.
Si deux $\flb$-représen\-ta\-tions irréductibles de $\G$
apparaissent dans $\rl(\pi)$, elles sont dépendantes.
\end{lemm}

\begin{proof}
Soit $\s$ un facteur irréductible de $\rl(\pi)$, 
et soit $\T$ l'ensemble des facteurs~irré\-ductibles $\tau$ de $\rl(\pi)$
tels que $\s$ et $\tau$ soient dépendants. 
On veut prouver que $\T$ est égal à~l'ensem\-ble de tous les facteurs
irréductibles de $\rl(\pi)$.
Supposons que ce ne soit pas le cas.
Nous~allons prouver qu'il existe
un $\zlb$-réseau $\L$ de l'espace $\V$ de $\pi$
tel que $\L$ soit stable par $\G$,
et tel qu'au\-cu\-ne~sous-re\-présentation irréductible de $\L\otimes\flb$
ne soit dans $\T$.

Pour cela, nous allons utiliser \cite{EH} Lemma 2.2.6.
Pour se convaincre que l'on peut~s'en servir, on observe que, 
d'après \cite{Vigb} II.4.7,
il existe une extension finie $\E$ de $\mathbb{Q}_{\ell}$ dans 
$\qlb$ et une $\E$-repré\-sen\-tation $\V_\E$ de $\G$ telles que~:
\begin{equation*}
\V\simeq\V_\E\otimes_{\E}\qlb.
\end{equation*}
L'espace $\V_\E$ contient un $\Oo_\E$-réseau $\M_\E$ stable par $\G$, 
et $\M_{\E}\otimes_{\Oo_{\E}}\zlb$ est un $\zlb$-réseau de $\V$
stable par $\G$.
Quitte à augmenter $\E$ si besoin,
on peut~même supposer que tous les sous-quotients~irré\-ductibles de
$\rl(\pi)$ sont définis sur $\kk_{\E}$,
le corps résiduel~de $\E$.
En particulier,
les éléments de~$\T$~sont tous~de la forme $\W\otimes_{\kk_\E}\flb$,
où $\W$ décrit un ensemble $\T_\E$ de
$\kk_\E$-représentations~irréductibles de $\G$.
L'anneau $\Oo_\E$ est un anneau de valuation discrète complet,
le module $\M_\E$ est séparé pour~la topologie $\p_\E$-adique
car il est libre sur $\Oo_\E$
et $\M_\E\otimes_{\Oo_\E}\kk_\E$ est une $\kk_\E$-représentation
de~$\G$ de~lon\-gueur finie d'après \cite{Vigw} Theorem 1
ou \cite{Vigb} II 5.11.b.
La $\E$-représen\-tation $\V_\E$ est donc une bonne représentation entière
au sens de \cite{EH} Definition 2.2.1,
et elle est admissible d'après \cite{Vigb} II.2.8.
Appliquant \cite{EH}~Lem\-ma 2.2.6, 
on obtient un $\Oo_\E$-réseau $\L_\E$ dans $\V_\E$, stable par $\G$, 
tel~qu'au\-cune sous-représentation irréductible de $\L_\E\otimes\kk_{\E}$
ne soit dans $\T_\E$.
Ainsi $\L=\L_\E\otimes\zlb$ est un $\zlb$-ré\-seau de $\V$ stable par $\G$
ayant la propriété voulue.

Soit maintenant $\rho$ une sous-représentation irréductible de
$\L\otimes\flb$.
Elle n'est pas dans $\T$,
donc $\rho$ et $\s$ ne sont pas dépendantes.
Il y a donc une décomposition~:
\begin{equation*}
\L\otimes\flb = \X_1 \oplus \X_2
\end{equation*}
où $\X_1$ est indécomposable et contient $\s$ mais pas $\rho$,
et où $\X_2$ contient $\rho$.
Tout sous-quotient~ir\-réductible de $\X_1$ est dans $\T$~: 
contradiction.
\end{proof}

\subsection{}
\label{parred}

En vue d'appliquer le lemme \ref{lemdep}, 
nous allons décrire la réduction mod $\ell$ d'une
$\qlb$-représenta\-tion~irré\-duc\-tible cuspidale
en\-tiè\-re de $\G$. 
Notons $\nu$ le caractère non ramifié ``valeur absolue de la norme~ré\-duite'' 
de $\G$.

\begin{prop}
\label{redcusp}
Soit $\mu$ une $\qlb$-représentation irréductible cuspidale entière de $\G$. 
\begin{enumerate}
\item 
Il existe une $\flb$-représentation irréductible cuspidale $\pi$ de $\G$ et 
un entier $a\>1$ tels que~:
\begin{equation}
\label{redcuspeq}
\rl(\mu) = \pi \oplus \pi\nu \oplus \dots \oplus \pi\nu^{a-1}.
\end{equation}
\item
Si $\epsilon(\pi)$ désigne le plus petit entier 
$i\>1$ tel que $\pi\nu^i$ soit isomorphe à $\pi$,
l'entier $a$ est égal soit~à $1$,
soit à $\epsilon(\pi)\ell^u$ pour un $u\>0$. 
\end{enumerate}
\end{prop}

\begin{proof}
La première partie est donnée par \cite{MSt} Theorème 3.15,
la seconde est donnée par \cite{MSt}~Lem\-me 3.19 et \cite{MSjl} 3.3.
\end{proof}

D'après \cite{MSjl} 3.3,
on a une autre description de $\epsilon(\pi)$.
Si $t(\pi)$ désigne le \textit{nombre de torsion}~de $\pi$, 
\ie le nombre de caractères non ramifiés $\chi$ de $\G$ tels que~$\pi$
soit isomorphe à $\pi\chi$,
et si $q$ est le cardinal~du corps résiduel de $\F$, alors~:
\begin{equation}
\label{defepstor}
\epsilon(\pi) = \text{ordre de $q^{t(\pi)}$ mod $\ell$}.
\end{equation}

Afin d'affiner la description de l'entier $a$
apparaissant dans \eqref{redcuspeq},
il nous faut~in\-troduire~des in\-variants~sup\-plé\-mentaires associés à la
représentation cuspidale $\pi$,
ce que nous faisons dans les paragraphes sui\-vants, à com\-mencer par le cas où 
$\pi$ est de niveau $0$. 

\subsection{}
\label{consnxi3}

Soit $\pi$ une représentation cuspidale de $\G$ de niveau $0$,
\ie que l'espace de ses~vecteurs invariants par le pro-$p$-sous-groupe 
$\N^1=1+\M_m(\p_\D)$ est non~nul.
D'après \cite{MSt}~Paragraphe~3.2, le~groupe~compact
maximal $\N^0=\GL_m(\Oo_\D)$ agit sur cet espace
\textit{via} une représenta\-tion du~groupe $\GL_m(\kk_\D)$,
de~di\-men\-sion finie car $\pi$ est admissible (\cite{Vigb} II.2.8),
et contenant une sous-représen\-tation irré\-ducti\-ble cuspidale~$\s$.
Selon~\cite{MSc}~Pa\-ra\-graphe 6.1,
on a le fait suivant.

\begin{enonce}{Fait}
\label{bitcoin0}
La représentation $\pi$ est supercuspidale si et seulement si 
$\s$ est supercuspidale.
\end{enonce}

Notons $\xi^0$ l'inflation de $\s$ à $\N^0$,
et $\N$ le normalisateur de la~clas\-se d'iso\-morphisme de $\xi^0$
dans $\G$.
D'après \cite{MSt} Paragraphe 3.1,
il y~a~un unique prolongement~de $\xi^0$ à $\N$, noté $\xi$,
tel que l'induite compacte de $\xi$ à $\G$ soit iso\-mor\-phe~à~$\pi$.

Soit $\mathscr{N}_\G(\N^0)$
le normalisateur de $\N^0$ dans $\G$,
qui est engendré par $\N^0$ et l'uniformisante $\w_\D$.
Le groupe $\N$ est compact mod le centre de $\G$ et on a~:
\begin{equation*}
\F^\times\N^0 \subseteq \N \subseteq \mathscr{N}_\G(\N^0).
\end{equation*}
D'après \cite{MSt} Lemme~3.2,
on a le fait suivant.

\begin{enonce}{Fait}
\label{bitcoin5}
La représentation de $\N$ sur l'espace des vecteurs
$\N^1$-invariants de $\pi$ 
est~iso\-mor\-phe~à~la somme directe~des
conjugués de $\xi$ sous $\mathscr{N}_\G(\N^0)$.
\end{enonce}

Notons $b=b(\pi)$ le nombre de conjugués de $\xi$ sous $\N$,
\ie l'indice de $\N$ dans $\mathscr{N}_\G(\N^0)$,
et $s=s(\pi)$ l'in\-di\-ce de~$\F^\times\N^0$ dans~$\N$,
deux entiers dont le produit vaut $d$.
On a~:
\begin{equation}
\label{genN}
\N = \langle \N^0,\w\rangle, \quad \w=\w_\D^{b}.
\end{equation} 
L'action de $\w_\D$~par con\-jugaison sur $\Oo_\D$ induit un
automorphisme de $\kk_\D$,
engendrant le groupe de~Galois $\Gal(\kk_\D/\kk_\F)$. 
Ainsi $b$ est le nombre de conjugués~de~$\s$ sous 
$\Gal(\kk_\D/\kk_\F)$
et $s$ est~l'ordre 
du stabilisateur de la classe d'isomorphisme de $\s$
dans $\Gal(\kk_\D/\kk_\F)$.
Ces deux entiers sont~indé\-pen\-dants du choix de $\s$.

Profitons-en pour introduire la définition sui\-van\-te,
qui nous sera utile dans la section \ref{sec7}.

\begin{defi}
\label{hasse}
L'\textit{invariant de Hasse} de $\D$ est l'unique entier $h\in\{1,\dots,d\}$
premier à $d$ tel que le $\kk_\F$-automorphisme de $\kk_\D$ induit par 
la conjugaison par $\w_\D$ soit égal à~:
\begin{equation*}
x \mapsto x^{q^h}
\end{equation*}
où $q$ est le cardinal de $\kk_\F$.
Il est indépendant du choix de $\w_\D$.
\end{defi}

\subsection{}
\label{flapflip}

Supposons maintenant que $\pi$ soit une représentation cuspidale
de niveau quelconque de~$\G$.~Il y~a,
d'après \cite{BK,SeSt1,MSt},
un sous-groupe ouvert compact $\J^{0}$ de $\G$
et une~re\-présen\-tation irréduc\-ti\-ble $\bl$ de $\J^{0}$
possédant les propriétés suivantes~:
\begin{enumerate}
\item 
Les représentations irréductibles de $\G$ dont la restriction à $\J^{0}$
contient $\bl$ sont exacte\-ment les représentations cuspidales de 
$\G$ inertiellement équivalentes à $\pi$.
\item 
Le groupe $\J^{0}$ a un unique pro-$p$-sous-groupe distingué maximal $\J^1$,
et la restriction de $\bl$ à $\J^1$ est un multiple d'une représentation
irréductible $\bn$. 
\item
La représentation $\n$ se prolonge en une représentation $\bk$ de $\J^{0}$,
et il y a une représentation irréductible $\bs$
de $\J^{0}$ triviale sur $\J^1$ telle que $\bl$ soit isomorphe à
$\bk\otimes\bs$.
\item
Il y a une extension finie $\E$ de $\F$ dans $\Mat_m(\D)$ telle que~: 
\begin{enumerate}
\item 
si $\B$ est le centralisateur de $\E$ dans $\Mat_m(\D)$,
alors $\J^{0}$ est égal à $(\J^{0}\cap\B^\times)\J^1$ et il existe un entier $r\>1$,
une $\E$-algèbre à division centrale $\C$ de degré réduit $c$ et 
un isomorphisme de $\E$-algèbres~:
\begin{equation}
\label{melon}
\phi:\B\simeq\Mat_r(\C),
\quad
rc = \frac{md} {[\E:\F]},
\quad
c = \frac{d}{(d, [\E:\F])},
\end{equation}
envoyant $\J^{0}\cap\B^\times$ sur le sous-groupe compact maximal standard
$\GL_r(\Oo_\C)$ et $\J^1\cap\B^\times$ sur~son unique
pro-$p$-sous-groupe distingué maximal $1+\Mat_r(\p_\C)$, 
\item
si $\ll$ est le corps résiduel de $\C$, 
et si l'on identifie le groupe~:
\begin{equation}
\label{melao}
\J^{0}/\J^1\simeq(\J^{0}\cap\B^\times)/(\J^1\cap\B^\times) 
\end{equation}
à $\GL_r(\ll)$
via un isomorphisme \eqref{melon}, 
la représentation $\bs$ est l'infla\-tion~d'une représentation 
cuspi\-dale $\s$ de $\GL_r(\ll)$, et
\item
la représentation de $\GL_r(\ll)$ sur l'espace $\Hom_{\J^1}(\n,\pi)$
définie par~:
\begin{equation*}
g\cdot f=\pi(g)\circ f\circ\k(g)^{-1},
\quad
g\in\J^0,
\quad
f\in \Hom_{\J^1}(\n,\pi),
\end{equation*}
est isomorphe à la somme directe des conjugués
de $\s$ sous $\Gal(\ll/\kk_\E)$.
\end{enumerate}
\end{enumerate}

De façon analogue au fait \ref{bitcoin0},
on a le fait suivant (\cite{MSc}~Pa\-ra\-graphe 6.1).

\begin{enonce}{Fait}
\label{bitcoin}
La représentation $\pi$ est supercuspidale si et seulement si 
$\s$ est supercuspidale.
\end{enonce}

Comme au paragraphe \ref{consnxi3}, on note $s(\pi)$
l'ordre du stabilisateur dans $\Gal(\ll/\kk_\E)$
de la classe d'isomorphisme de $\s$.
On définit aussi~:
\begin{equation}
\label{matchaa}
f(\pi) = \frac {n} {e_{\E/\F}},
\quad
q(\pi) = q^{f(\pi)}.
\end{equation}
D'après \cite{MSt} (3.6), on a la propriété suivante, liant les invariants
$s(\pi)$, $f(\pi)$ et le nombre~de~tor\-sion $t(\pi)$ défini au paragraphe 
\ref{parred}.
On note $\ell$ l'exposant caractéristique de $\R$.

\begin{enonce}{Fait}
\label{matchab}
Il y a un entier $v\>0$ tel que $f(\pi)=t(\pi)s(\pi)\ell^v$.
\end{enonce}

Comme en niveau $0$, il existe un unique prolongement $\l$
de $\bl$ au normalisateur dans $\G$ de sa classe d'isomorphisme,
dont l'induite compacte à $\G$ soit isomorphe à $\pi$, 
mais nous n'utiliserons pas ce fait. 

\begin{rema}
\label{etiennelousteau}
Si $\pi$ est de niveau $0$,
on a $\E=\F$, $\J^{0}=\N^0$, $\J^1=\N^1$ et $\bl=\xi^0$
avec~les~no\-tations du paragraphe \ref{consnxi3}.
L'entier $f(\pi)$ est donc simplement égal à $n$ dans ce cas,
et le nombre de torsion $t(\pi)$,
premier à $\ell$ par définition,
est le plus grand diviseur de $mb(\pi)$ premier à $\ell$.
\end{rema}

\subsection{}

Il sera utile de préciser les faits \ref{bitcoin0} et \ref{bitcoin}. 
Le résultat suivant vient de \cite{MSc} Section 6~(et 
en particulier de \cite{MSc} Théorème 6.14).

\begin{prop} 
\label{cuspscusp5}
Soit $\pi$ une représentation irréductible cuspidale de $\G$.
\begin{enumerate}
\item 
Il y a un unique diviseur $k=k(\pi)$ de $m$ et une représentation
irréductible supercuspidale $\rho$ de $\GL_{m/k}(\D)$ tels que~:
\begin{equation}
\label{cuspscusp5eq}
\scusp(\pi) = \rho+\rho\nu+\dots+\rho\nu^{k-1}.
\end{equation}
\item 
L'induite parabolique $\rho\times \rho\nu\times\dots\times\rho\nu^{k-1}$
ne contient aucun 
sous-quotient irréductible~cus\-pidal non isomorphe à $\pi$,
et elle contient $\pi$ avec mutiplicité $1$.
\end{enumerate}
\end{prop}

D'après \cite{MSf} Théorèmes 2.4 et 2.5, on a un
analogue fini de la proposition \ref{cuspscusp5}.

\begin{prop}
\label{cuspscusp5fini}
Soit $\s$ une représentation irréductible cuspidale de $\GL_r(\kk_\C)$.
\begin{enumerate}
\item 
Il existe un unique~di\-viseur $k(\s)$ de $m$ et une 
$\flb$-représentation~ir\-ré\-ductible
supercuspi\-dale
$\a$~de~$\GL_{m/k(\s)}(\kk_\D)$, unique~à iso\-morphis\-me~près,
tels que~:
\begin{equation*}
\scusp(\s)=\a+\dots+\a.
\end{equation*}
\item
L'induite parabolique $\a\times\dots\times\a$ ne contient aucun 
sous-quotient irréductible~cus\-pidal non isomorphe à $\s$,
et elle contient $\s$ avec mutiplicité $1$.
\end{enumerate}
\end{prop}

D'après \cite{MSjl} Lemme 3.2, 
on a le fait suivant,
qui généralise les faits \ref{bitcoin0} et \ref{bitcoin}. 

\begin{enonce}{Fait}
\label{bitcoink} 
Si $\pi$ et $\s$ sont comme au paragraphe \ref{flapflip},
alors $k(\pi)=k(\s)$.
\end{enonce}

\subsection{}
\label{paradivs}

Précisons maintenant les résultats du paragraphe \ref{parred}.
Soient $\mu$, $\pi$ des représentations~com\-me
dans la proposition \ref{redcusp}, 
et écrivons le support supercuspidal de $\pi$ sous la forme
\eqref{cuspscusp5eq}.
D'après \cite{MSt} \'Equation (3.9) et
\cite{MSc}~Pro\-po\-sition 6.4,~Co\-rol\-lai\-re~6.12,
on a le fait suivant.

\begin{enonce}{Fait} 
\label{adivs}
\label{formek} 
On a $s(\pi)=as(\mu)$ et $s(\rho)=s(\pi)$. 
\end{enonce}

Le résultat suivant,
que nous n'énonçons que dans le cas où la $\flb$-représentation $\pi$~est 
supercuspidale,
donne une réciproque à la proposition \ref{redcusp}.
  
\begin{prop}
\label{redcuspb}
Soit $\pi$ une $\flb$-représentation supercuspidale de $\G$
et soit un entier $a\>1$.
Pour qu'il y ait une $\qlb$-représentation irréductible cuspidale entière
de $\G$ dont la réduction mod~$\ell$ con\-tienne $\pi$ et soit de longueur $a$, 
il faut et suffit qu'une des conditions suivantes~soit~vé\-ri\-fiée~: 
\begin{enumerate}
\item 
ou bien $a=1$, 
\item
ou bien
$\epsilon(\pi)$ divise $s(\pi)$ et
$a$ est un diviseur de $s(\pi) $ de la forme 
$a=\epsilon(\pi)\ell^u$ pour~un $u\>0$.
\end{enumerate}
\end{prop}

\begin{proof}
Voir \cite{MSc} Theorème 6.11 si $a=1$, 
et \cite{MSjl} Proposition 1.6 si $a\neq1$.
\end{proof}

Comme $\epsilon(\pi)$ est l'ordre de 
$q^{t(\pi)}$ mod $\ell$ d'après \eqref{defepstor},
le fait que $\epsilon(\pi)$~divi\-se $s(\pi)$ équivaut
à~ce que $\ell$ divise $q^{t(\pi)s(\pi)}-1$, 
ce qui,
d'après le fait \ref{matchab}
et compte tenu de \eqref{matchaa}, équivaut à~: 
\begin{equation*}
\label{lemoutonenrage}
\text{$\ell$ divise $q(\pi)-1$.}
\end{equation*}
On en déduit les corollaires suivants. 

\begin{coro}
\label{coromal}
Soit $\pi$ une $\flb$-représentation irréductible supercuspidale de $\G$
et soit~$j\in\ZZ$.
Pour qu'il y ait une $\qlb$-représentation cuspidale
entière de $\G$ dont la
réduction~mod $\ell$~con\-tien\-ne à la fois $\pi$ et $\pi\nu^j$,
il faut et suffit que 
$\pi$ et $\pi\nu^j$ soient isomorphes
ou que $\ell$ divise $q(\pi)-1$.
\end{coro}

\begin{proof}
La condition est nécessaire~:
si une telle représentation existe,
et si 
$\pi$ et~$\pi\nu^j$ ne sont pas isomorphes,
la proposition \ref{redcusp} assure que~la~lon\-gueur $a$ de sa réduction
mod $\ell$ est un multiple de $\epsilon(\pi)$, et le résultat suit
du fait \ref{adivs}.

Inversement, 
si les représentations $\pi$ et~$\pi\nu^j$ sont isomorphes,
le résultat suit de la proposition \ref{redcuspb} appliquée avec $a=1$.
Si $\pi$ et $\pi\nu^j$ ne sont pas isomorphes
mais $\epsilon(\pi)$ divise $s(\pi)$, 
le résultat suit 
de la proposition \ref{redcuspb} appliquée avec $a=\epsilon(\pi)$.
\end{proof}

\begin{coro}
\label{coromal2q3}
Soit $\pi$ une $\flb$-représentation supercuspidale de $\G$. 
\begin{enumerate}
\item 
Si $\ell$~divi\-se $q(\pi)-1$,
alors, pour tout $j\in\ZZ$, les représentations $\pi$ et
$\pi\nu^j$ sont dépendantes. 
\item 
Supposons que $\ell$~ne divi\-se pas $q(\pi)-1$,
et soit $\mu$ une $\qlb$-représentation cuspidale entière de $\G$
dont la~ré\-duc\-tion mod $\ell$ contient $\pi$.
Alors $\rl(\mu)$~est irréductible,
isomorphe à $\pi$.
\end{enumerate}
\end{coro}

\begin{proof}
Le point (1) est une conséquence du corollaire \ref{coromal} 
et du lemme \ref{lemdep}.
Le~point (2) est une conséquence de la proposition \ref{redcusp}(2)
et du fait \ref{adivs}.
\end{proof}

\begin{rema}
D'après \cite{MSt} Proposition 4.40 et Lemme 4.45,
$\ell$~divi\-se $q(\pi)-1$ si et seu\-le\-ment si
l'induite parabolique $\pi\times\pi$ est réductible.
\end{rema}

\subsection{}
\label{defbepi}

Avant de mettre fin à cette section,
discutons de la réduction mod $\ell$ d'une $\qlb$-représentation irréductible 
entière quelconque de $\G$.
Introduisons la définition suivante. 

\begin{defi}
\label{Notation C(pi)}
Soit $\pi$ une $\flb$-représentation irréductible de $\G$,
dont on écrit~: 
\begin{equation*}
\scusp(\pi) = \rho_1+\dots+\rho_r
\end{equation*}
le support supercuspidal.
On note $\B(\pi)$ l'ensemble des représentations irréductibles $\pi'$
de $\G$~dont le support supercuspidal est de la forme~: 
\begin{equation*}
\scusp(\pi') = \rho_1\nu^{j_1}+\dots+\rho_r\nu^{j_r},
\quad
j_1,\dots,j_r\in\ZZ,
\end{equation*}
où, pour chaque $k=1,\dots,r$, l'entier $j_k$ est nul 
si $\ell$ ne divise pas $q(\rho_k)-1$.
\end{defi}

\begin{rema}
\label{Bsingleton}
Pour que $\B(\pi)$ contienne des représentations de supports 
supercuspidaux différents, 
il faut et suffit donc qu'il y ait un $k\in\{1,\dots,r\}$ tel que
$\epsilon(\rho_k)\neq1$ et $\ell$ divise $q(\rho_k)-1$.

Si $\G$ est déployé,
\ie si $d$ est égal à $1$,
cela ne se produit jamais~:
on a $s(\rho_k)=1$ pour tout $k$,
donc le fait que $\ell$ divise $q(\rho_k)-1$
implique que $\epsilon(\rho_k)=1$.

Si $\pi$ est supercuspidale de niveau $0$,
cela se produit si $\ell$ divise $q^n-1$ mais pas
$q^{n/s(\rho)}-1$,
sachant que $s(\rho)$ divise $d$ et est premier à $m$
(voir \cite{MSjl} Corollaire 3.9).
En particulier,
pour~le~ca\-ractère~trivial de $\D^\times$,
cela se produit si $\ell$ divise $q^d-1$ mais pas $q-1$.
\end{rema}

La relation $\pi'\in\B(\pi)$
équivaut à $\B(\pi')=\B(\pi)$.
C'est une relation d'équivalence
sur l'ensemble des $\flb$-représentations~irré\-duc\-ti\-bles de $\G$.

\begin{prop}
\label{blocindec}
Soient $\pi$ et $\pi'$ des $\flb$-représentations~ir\-réductibles de $\G$. 
Pour que~les~en\-sem\-bles $\B(\pi)$ et $\B(\pi')$ soient égaux, 
il faut et suffit qu'il existe une $\qlb$-représentation 
irréductible en\-tière de $\G$ 
dont la réduction mod $\ell$ contienne à la fois $\pi$ et $\pi'$.
\end{prop}

\begin{proof}
Supposons d'abord que $\pi'\in\B(\pi)$.
Pour chaque $i=1,\dots,r$,
choisissons une $\qlb$-re\-pré\-sen\-tation cuspidale entière $\mu_i$ de
$\GL_{m_i}(\D)$ dont la réduction mod $\ell$ contienne à la fois $\rho_i$
et $\rho_i\nu^{j_i}$, dont l'existence est assurée par le corollaire
\ref{coromal}.
Soient maintenant $\h_1,\dots,\h_r$ des caractères non ramifiés de
$\GL_{m_1}(\D),\dots, \GL_{m_r}(\D)$ à valeurs dans $1+\mathfrak{m}$
(où $\mathfrak{m}$ est l'idéal~maxi\-mal~de $\zlb$)
et considérons l'induite parabolique~:
\begin{equation}
\label{chifoumi}
\mu_1\h_1 \times \dots \times \mu_r\h_r.
\end{equation}
Par construction, sa réduction mod $\ell$ contient à la fois $\pi$ et $\pi'$,
l'induction parabolique étant~com\-patible à la réduction mod $\ell$
(voir \cite{MSc} 1.2.3 ou \cite{Vigb} II.4.14).
Nous allons prouver que,
pour un choix convenable des $\h_i$,
cette induite est irréductible, ce qu'on va prouver par récurrence~sur~$r$. 

Posons $l=r-1$,
et supposons que $l\>1$ et que l'induite
$\mu_1\h_1 \times \dots \times \mu_{l}\h_{l}$ soit irréductible.
Pour que \eqref{chifoumi}~le soit,
il faut et suffit
(d'après \cite{MSc} Corollaire 7.32 par exemple)
de choisir~$\h_r$ de sorte que,
pour tout $i<r$,~l'in\-duite 
$\mu_i\h_i\times \mu_r\h_r$ soit irréductible,
\ie (voir \cite{VSU0} Section~4)
que $\mu_r\h_r$~n'ap\-partien\-ne~pas à la réunion~:
\begin{equation*}
\left\{\mu_1\h_1\nu^{s(\mu_1)},\mu_1\h_1\nu^{-s(\mu_1)}\right\}
\cup\dots\cup
\left\{\mu_l\h_l\nu^{s(\mu_l)},\mu_l\h_l\nu^{-s(\mu_l)}\right\} 
\end{equation*}
(les entiers $s(\mu_i)$ et $t(\mu_i)$ sont
introduits au paragraphe \ref{parred}), \ie que~:
\begin{equation*}
\label{cond22}
\left(\h_r^{\phantom{1}}\chi_i^{\phantom{1}}\h_i^{-1}\nu^{s(\mu_i)}\right)^{t(\mu_i)}
\neq 1
\quad\text{et}\quad
\left(\h_r^{\phantom{1}}\chi_i^{\phantom{1}}\h_i^{-1}\nu^{-s(\mu_i)}\right)^{t(\mu_i)}
\neq 1
\end{equation*} 
à chaque fois que $\mu_r$ et $\mu_i$ sont inertiellement équivalents, 
$\chi_i$ étant un caractère non ramifié~tel~que $\mu_r$ soit isomorphe à
$\mu_i\chi_i$. 
Cette condition n'interdit qu'un nombre fini de valeurs pour $\h_r$,
alors que $1+\mathfrak{m}$ est infini.

Supposons maintenant que $\pi$ et $\pi'$ apparaissent dans la réduction mod 
$\ell$ d'une $\qlb$-représen\-tation irréductible en\-tière de $\G$, 
dont on note $\mu_1+\dots+\mu_u$ le support cuspidal,
$\mu_i$ étant une~$\qlb$-re\-pré\-sen\-tation cuspidale entière de
$\GL_{n_i}(\D)$ pour chaque $i=1,\dots,u$, 
avec $n_1+\dots+n_u=m$.
Les représentations
$\pi$ et $\pi'$ apparaissent donc dans la réduction mod $\ell$ de 
l'induite parabolique~:
\begin{equation*}
\label{prechifoumi}
\mu_1 \times \dots \times \mu_u. 
\end{equation*}
Pour chaque $i$,
écrivons $\r_\ell(\mu_i)$ sous la forme
$\pi_{i}\oplus \pi_i\nu \oplus \dots\oplus \pi_{i}\nu^{a_i-1}$
d'après la proposition \ref{redcusp},
et écrivons $\scusp(\pi_i)$ sous la forme
$ \rho_i+\rho_i\nu+\dots+\rho_i\nu^{k_i-1}$
d'après la proposition \ref{cuspscusp5}. 
Il y a des
en\-tiers $c_i^{\phantom{'}}$ et $c_i'$ dans $\{0,\dots,a_i-1\}$ tels que~:
\begin{equation*}
\scusp(\pi) = \sum\limits_{i=1}^{u} \sum\limits_{j=0}^{k_i-1}
\rho_i\nu^{j+c_i^{\phantom{'}}},
\quad
\scusp(\pi') = \sum\limits_{i=1}^{u} \sum\limits_{j=0}^{k_i-1} \rho_i\nu^{j+c'_i}.
\end{equation*}
En outre,
si $\ell$ ne divise pas $q(\rho_i)-1$, 
on a $a_i=1$ d'après la proposition \ref{redcuspb},
ce qui entraîne que $c_i'=c_i^{\phantom{'}}$.
Par conséquent, on a $\pi'\in\B(\pi)$.
\end{proof}

Par ailleurs,
le lemme \ref{lemdep} assure que des représentations irréductibles $\pi$, 
$\pi'$ apparaissant~dans la réduction mod $\ell$ d'une $\qlb$-représentation 
irréductible en\-tière de $\G$
--- ou, ce qui est équivalent d'après la proposition \ref{blocindec},
telles que $\pi'\in\B(\pi)$ --- 
sont dépendantes. 
L'objet des sections \ref{sec4} à \ref{sec6} est de prouver la réciproque
(voir le théorème \ref{Prop 3.2.4} et la proposition \ref{blocindecfinal}).

\section{Le cas cuspidal de niveau zéro} 
\label{sec4}

Cette section est consacrée à la preuve du résultat suivant.

\begin{prop}
\label{ext1cuspniveau0}
Soient $\pi$ et $\pi'$ des $\flb$-représentations cuspidales de niveau $0$
de $\G$.
Suppo\-sons que $\Ext^1_\G(\pi,\pi')$ soit non nul.
Alors $\pi'\in\B(\pi)$.
\end{prop}

Le cas des représentations cuspidales de niveau quelconque fera l'objet de la
section suivante. 

\subsection{}
\label{consnxi}

Dans cette section,
on fixe une représentation cuspidale $\pi$ de niveau $0$ de $\G$,
à coefficients~dans un corps algébriquement clos $\R$
de caractéristique~différen\-te de $p$
qu'on suppose quelconque~jus\-qu'au paragraphe \ref{subsection 3.1.2}.
Comme au pa\-ra\-graphe \ref{consnxi3} dont on reprend les notations,
on~fixe~une~paire $(\N,\xi)$ dont l'induite à $\G$ est iso\-mor\-phe à $\pi$.
En particulier,
la restriction de~$\xi$ à~$\N^0$~est~ir\-réduc\-tible,
et est l'inflation d'une représentation cuspidale $\s$ de $\GL_m(\kk_\D)$.
L'ordre du stabilisateur~de $\s$ sous l'action de $\Gal(\kk_\D/\kk_\F)$,
égal à l'indice de $\F^\times\N^0$ dans $\N$,
est noté $s=s(\pi)$.

Par ailleurs, 
d'après la proposition \ref{cuspscusp5}, 
il y a un unique diviseur $k=k(\pi)$ de $m$
et une~repré\-sen\-tation irréductible 
supercuspidale $\rho$~de $\GL_{m/k}(\D)$ 
tels que~:
\begin{equation*}
\scusp(\pi) = \rho + \rho\nu + \dots + \rho\nu^{k-1},
\end{equation*}
et toute représentation cuspidale de $\G$ de même sup\-port su\-per\-cuspidal
que $\pi$ est isomorphe à~$\pi$.

\subsection{}

Prouvons d'abord le lemme général suivant. 
Si $\V$ est une représentation de $\G$,
on notera $\V^{\N^1}$ l'espace de ses vecteurs invariants par $\N^1$.
Le sous-groupe $\N^1$ étant distingué dans $\N$,
cet espace définit une représentation de~$\N$ triviale sur $\N^1$.

\begin{lemm}
\label{decdebase2}
Soit $\V$ une représentation de $\G$. 
Pour tout $i\>0$,
on a un isomorphisme~:
\begin{equation*}
\Ext^i_\G(\pi,\V) \simeq \Ext^i_\N(\xi,\V^{\N^1})
\end{equation*}
de $\R$-espaces vectoriels.
\end{lemm}

\begin{proof}
Notons $\X$ la restriction de $\V$ à $\N$.
Par adjonction (voir \cite{Vigb} A.2),
il y a un iso\-morphisme~:
\begin{equation*}
\Ext^i_\G(\pi,\V) \simeq \Ext^i_\N(\xi,\X) 
\end{equation*}
de $\R$-espaces vectoriels.
Comme $\N^1$ est un pro-$p$-groupe 
et $p$ est inversible dans $\R$, 
il existe~une uni\-que~décomposition $\X=\Y\oplus\Z$,
où $\Y$ est la représentation de $\N$ sur $\V^{\N^1}$ 
et $\Z$ est une~repré\-sentation de $\N$ dont aucun composant irréductible
n'est invariant par $\N^1$. 
On a donc~:
\begin{equation*}
\Ext^i_\N(\xi,\X) \simeq \Ext^i_\N(\xi,\Y) \oplus \Ext^i_\N(\xi,\Z)
\end{equation*}
et le résultat vient de ce que l'espace $\Ext^i_\N(\xi,\Z)$ est nul,
car $\xi$ est trivial sur le pro-$p$-groupe $\N^1$
et aucun composant de $\Z$ ne l'est.
\end{proof}

\subsection{}
\label{floupifloup}

Soit $\pi'$ une représentation irréductible cuspidale de $\G$ telle que 
l'espace $\Ext_\G^i(\pi,\pi')$ soit~non nul pour un $i\>0$.

\begin{lemm}
\label{decdebase25}
Il existe un caractère non ramifié $\chi$ de $\G$ tel que
$\pi'$ soit~iso\-morphe à $\pi\chi$. 
\end{lemm}

\begin{proof}
Le théorème \ref{decbvss} assure que
le support supercuspidal de $\pi$ est inertiel\-lement équivalent à celui de 
$\pi'$.
Si l'on écrit ce dernier sous la forme
$\tau + \tau\nu + \dots + \tau\nu^{l-1}$,
où $l$ est un diviseur de $m$~et~$\tau$ une~représentation irréductible 
supercuspi\-da\-le~de $\GL_{m/l}(\D)$, 
on trouve~que~$l$ est~égal~à~$k$ et que
$\tau$ et $\rho$
sont inertiellement équivalentes.
Il y a donc un caractère non ramifié~$\chi$ de $\G$ tel que $\pi'\chi^{-1}$
et $\pi$ aient le même support supercuspidal.
\end{proof}

Fixons donc un caractère non ramifié $\chi$ de $\G$ tel que
$\Ext^i_{\G}(\pi , \pi \chi)$ soit non nul. 
La~représenta\-tion $\pi$ étant isomorphe à l'induite compacte
de $\xi$ à $\G$, 
il~s'ensuit que $\pi\chi$ est isomorphe à~l'indui\-te~compacte
de $\xi\chi$ à $\G$.

(Les représentations $\pi$ et $\pi\chi$ ayant en outre 
même caractère central
d'après le lemme \ref{lemmecc},
on a même $\chi^n=1$,
mais nous n'aurons pas besoin de cette précision.)

Selon le fait \ref{bitcoin5},
la représentation de $\N$ sur l'espace des vec\-teurs $\N^1$-invariants
de $\pi\chi$~est isomorphe à la somme directe des conjugués
de $\xi\chi$ sous le normalisateur de $\N^0$ dans $\G$.
Appliquant le lemme \ref{decdebase2} à la représentation $\pi\chi$,
on en déduit le résultat suivant. 

\begin{prop}
\label{decdebase3}
Il y a un isomorphisme de $\R$-espaces vectoriels~:
\begin{equation*}
\label{decextpipichi}
\Ext^i_\G(\pi,\pi\chi) \simeq \bigoplus\limits_{u} \Ext^i_\N(\xi,\xi^{u}\chi)
\end{equation*}
où $u$ décrit un système de représentants de $\mathscr{N}_\G(\N^0)$ 
mod $\N$. 
\end{prop}

Nous pouvons donc reformuler la proposition \ref{ext1cuspniveau0}
de la façon suivante~:
si $\chi$ est un $\flb$-caractère non ramifié de $\G$ tel que l'espace 
$\Ext^1_\N(\xi,\xi^{u}\chi)$ soit non nul pour un $u\in\mathscr{N}_\G(\N^0)$,
il y a un~entier $j\in\ZZ$,
qu'on peut prendre égal à $0$ si $\ell$ ne divise pas $q^n-1$,
tel que les caractères $\chi$ et $\nu^j$~coïn\-cident sur~$\N$.
Nous~le prouverons dans les paragraphes \ref{merlinlenchanteur} à 
\ref{utherpendragon},
et nous prouverons au~passa\-ge~(corollaire \ref{decdebase3coro})
qu'un tel $u$ doit être égal à $1$.

Avant d'aller plus loin,
nous allons prouver que~$\xi$ 
admet une en\-ve\-lop\-pe projective. 

\subsection{}
\label{subsection 3.1.2}

Plus généralement, 
nous allons prouver que toute représentation irréductible de $\N$ 
triviale sur $\N^1$ admet une en\-ve\-lop\-pe projective dans $\rep_\R(\N)$.

Soit $\mp$ une représentation irréductible de $\N$ triviale sur $\N^1$.
Une \textit{enve\-lop\-pe projective} de~$\mp$ dans $\rep_\R(\N)$
est une représentation projective $\P$ de cette~caté\-gorie, 
munie d'un mor\-phis\-me~sur\-jec\-tif
$f:\P\to\mp$ trivial sur toute sous-représentation propre de $\P$
(voir par exemple \cite{Vigb} I.A.4).
La représentation $\mp$~étant irréductible,
le noyau de $f$ est alors l'unique 
sous-représentation~propre maximale de $\P$,
de sorte que $f$ est déterminé à un scalaire non nul près. 
Aussi omettra-t-on~la~ré\-férence à $f$
et dira-t-on par abus de langage que $\P$ est une
enveloppe~projective de~$\mp$.~Réci\-pro\-que\-ment,
si $\P$ est une représentation projective de $\rep_\R(\N)$
ayant une unique sous-représen\-ta\-tion~propre maximale $\M$,
et si le quotient de $\P$ par $\M$ est~isomorphe à $\mp$, 
alors $\P$ 
est une~enve\-lop\-pe~projective de $\mp$.
Prouvons d'abord le lemme général suivant. 

\begin{lemm}
\label{decdebase}
Soit $\X$ une représentation de longueur finie de $\N$.
Il y a des sous-repré\-sen\-tations $\X_1$ et $\X_0$ de $\X$,
uniques, telles que $\X=\X_1\oplus\X_0$ et~:
\begin{enumerate}
\item 
la représentation $\X_1$ est invariante par $\langle\N^1,\varpi_\F\rangle$, 
\item
aucun des composants irréductibles de $\X_0$ n'est invariant par 
$\langle\N^1,\varpi_\F\rangle$.  
\end{enumerate}
\end{lemm}

\begin{proof}
Comme dans la preuve du lemme \ref{decdebase2},
du fait que $\N^1$ est un pro-$p$-groupe et que $p$ est inversible dans $\R$, 
il~y~a une unique décomposition $\X=\Y\oplus\Z$
en sous-représentations telles que $\Y=\X^{\N^1}$
et aucun des composants irréductibles de $\Z$ ne soit invariant par $\N^1$. 

Puis, comme au lemme \ref{lemmecc},
il existe des scalaires $z_1,\dots,z_r\in\R^\times$ deux à deux distincts et une
décomposition~: 
\begin{equation*}
\Y = \Y_1 \oplus\dots\oplus\Y_r
\end{equation*}
telle que,
pour chaque $i\in\{1,\dots,r\}$, 
l'uniformisante $\w_\F$ agisse sur chaque composant irréducti\-ble
de $\Y_i$ par le scalaire $z_i$.
Si aucun des $z_i$ n'est égal à $1$,
on pose $\X_1=\{0\}$ et $\X_0=\X$.
Sinon,
on peut renuméroter de sorte que $z_1=1$,
et on pose $\X_1=\Y_1$ 
et $\X_0=\Y_2 \oplus\dots\oplus\Y_r\oplus\Z$.
\end{proof}

On sait que~$\N$ est engendré par $\N^0$~et l'élément $\w$ 
défini au paragraphe \ref{consnxi3}.
Comme $\w^s=\w_\F$, 
le quotient~: 
\begin{equation*}
\label{defNF}
\NF = \N / \langle\N^1,\varpi_\F \rangle
\end{equation*}
est un groupe fini.
D'autre part,
la restriction de $\mp$ à $\F^\times$,
le centre de $\N$,
est un multiple d'un caractère $\omega_{\mp}$.
Fixons un
caractère $\omega$ de $\N$ trivial sur $\N^0$ tel que
$\omega(\w_\F) = \omega_{\mp}(\varpi_\F)$ et posons~:
\begin{equation*}
\mp_1 = \mp\omega^{-1}. 
\end{equation*}
C'est une représentation de $\N$ triviale sur
$\langle\N^1,\varpi_\F \rangle$.
Elle peut donc être vue comme~une~repré\-sentation de $\NF$.
D'après \cite{Serre} Proposition 41,
elle a une enveloppe projective $\P_1$
dans la catégorie des~représentations~de longueur finie de 
$\NF$,
unique à isomorphisme près.
Elle est~indécomposa\-ble
et~sa~res\-triction à $\F^\times$ est un multiple du caractère 
$\omega_{\mp}\omega^{-1}$. 

\begin{lemm}
\label{canard}
La représentation $\P_1$ vue comme représentation de $\N$
est une enve\-loppe projective de $\mp_1$ dans $\rep_\R(\N)$. 
\end{lemm}

\begin{proof}
Il suffit de prouver que $\P_1$ est projective dans $\rep_\R(\N)$.
Considérons le~dia\-gramme~:
\begin{equation*}
\xymatrix{ & \X \ar@{->>}[d]^{g} & \\ \P_1 \ar@{->}[r]_{v} & \Y }
\end{equation*}
dans la catégorie des représentations de longueur finie de $\N$.
\'Ecrivons~:
\begin{equation*}
\X=\X_1 \oplus \X_0,
\quad
\Y=\Y_1 \oplus \Y_0,
\end{equation*}
les décompositions de $\X$ et $\Y$ données par le lemme \ref{decdebase}. 
Comme $\P_1$ est triviale sur $\langle\N^1,\varpi_\F\rangle$,~on a
$\textup{Im}(v) \subseteq \Y_1$,
\ie que $v=v_1\oplus 0$.
\'Ecrivons $g=g_1\oplus g_0$ avec~:
\begin{equation*}
g_1\in\Hom_\N(\X_1,\Y_1), \quad g_0\in\Hom_\N(\X_0,\Y_0).
\end{equation*}
Par projectivité de $\P_1$ dans la catégorie des 
$\R$-représentations de $\NF$,
il y a un morphisme $w_1$~de $\P_1$ dans $\X_1$ tel que
$g_1\circ w_1=v_1$.
Posant $w=w_1\oplus0$,
on obtient $g\circ w=v$,~ce qui prouve que $\P_1$
est projective dans $\rep_\R(\N)$. 
\end{proof}

Montrons à présent un lemme qui
affirme que la projectivité résiste à la torsion par un~ca\-rac\-tère.  

\begin{lemm}
\label{Lemme 3.2.3}
Soit $\Q$ une représentation projective de $\N$, et soit $\psi$ un
caractère de $\N$.
\begin{enumerate}
\item 
La représentation $\Q\psi$ est une représentation projective de $\N$. 
\item
Si $\Q$ est une enveloppe projective d'une représentation 
irréductible $\W$ de $\N$, 
alors $\Q\psi$ est une enveloppe projective de $\W\psi$.
\end{enumerate}
\end{lemm}

\begin{proof}
Considérons le diagramme~:
\begin{equation*}
\xymatrix{ & \X \ar@{->>}[d]^{g} & \\ \Q\psi \ar@{->}[r]_{v} & \Y }
\end{equation*}
dans $\rep_\R(\N)$.
Tordant par $\psi^{-1}$,
et $\Q$ étant projective,
on a un morphisme $w\in\Hom_\N(\Q,\X\psi^{-1})$
tel que $g\circ w=v$.
L'assertion (1) se déduit du fait que
$\Hom_\N(\Q, \X\psi^{-1})=\Hom_\N(\Q\psi,\X)$.~L'as\-ser\-tion
(2) en découle immédiatement. 
\end{proof}

\begin{prop}
\label{Propo 3.1.4} 
La représentation $\mp$ admet une enveloppe projective 
dans $\rep_\R(\N)$.
\end{prop}

\begin{proof}
D'après les lemmes \ref{canard} et \ref{Lemme 3.2.3}, 
la représentation $\P_1\omega$
est une enveloppe projective de $\mp$ dans $\rep_\R(\N)$.
\end{proof}

Par unicité de l'enveloppe projective à isomorphisme près, 
la classe~d'iso\-morphisme de $\P_1\omega$ est
indépendante du choix du caractère $\omega$ qui a servi à sa 
construction.

\subsection{}
\label{fiduciel}

On reprend les notations du paragraphe précédent,
en supposant que $\R$ est le corps $\flb$.
On~a donc une $\flb$-représentation irréductible $\mp$
de $\N$ triviale sur $\N^1$.

\begin{prop}
\label{propVigneras} 
Soit $\P$ l'enveloppe projective de $\mp$ dans $\rep_{\flb}(\N)$.
\begin{enumerate}
\item 
Il y a, 
à isomorphisme près, 
une unique $\zlb$-représentation projective 
$\tP$ dans $\rep_{\zlb}(\N)$ 
telle que $\tP\otimes\flb$ soit isomorphe à $\P$.
\item
Pour toute $\qlb$-représentation irréductible entière $\xp$ de $\N$, 
la multiplicité de $\xp$ dans $\tP\otimes\qlb$
est égale à celle de $\mp$ dans $\rl(\xp)$. 
\end{enumerate}
\end{prop}

\begin{proof}
Nous reprenons les notations du paragraphe précédent.
D'après la proposition 42 et le paragraphe 15.4 de \cite{Serre},
on a les faits suivants~:
\begin{itemize}
\item 
il y a, 
à isomorphisme près, 
une unique $\zlb$-représentation projective 
$\tP_1$ dans $\rep_{\zlb}(\NF)$ 
telle que $\tP_1\otimes\flb$ soit isomorphe à $\P_1$, et 
\item
pour toute $\qlb$-re\-pré\-sentation irréductible $\xp_1$ de $\NF$,
la multiplicité de $\xp_1$ dans $\tP_1\otimes\qlb$
est égale à
la multiplicité de $\mp_1$ dans $\rl(\xp_1)$.
\end{itemize} 
Fixons un $\zlb$-caractère $\widetilde{\omega}$ de $\N$ tel que
$\rl(\widetilde{\omega})=\omega$ et posons~:
\begin{equation}
\label{girafe}
\tP= \tP_1\widetilde{\omega}.
\end{equation}
C'est une représentation projective
d'après le lemme \ref{Lemme 3.2.3},
et $\tP\otimes\flb$ est isomor\-phe à $\P$.
Si $\Q$ est une $\zlb$-représentation projective
dans $\rep_{\zlb}(\N)$ ayant la~mê\-me propriété,
alors, par projectivité, tout isomorphisme entre
$\tP\otimes\flb$ et $\Q\otimes\flb$
se relève en un~iso\-mor\-phisme entre $\tP$ et $\Q$.

Afin de prouver la seconde partie de la proposition,
nous aurons besoin du lemme suivant.

\begin{lemm}
\label{huysmans}
Soit $\xp$ une $\qlb$-représentation irréductible entière du groupe $\N$. 
Les assertions sui\-van\-tes sont équivalentes~:
\begin{enumerate}
\item 
La représentation $\xp$ est triviale sur $\N^1$.
\item
La représentation $\rl(\xp)$ est triviale sur $\N^1$.
\item
La représentation $\rl(\xp)$ contient un facteur irréductible trivial sur 
$\N^1$.
\end{enumerate}
\end{lemm}

\begin{proof}
Bien sûr, 
(1) implique (2), lui-même impliquant (3).
Prouvons que (3) impli\-que (1).
Le groupe $\N^1$ étant distingué dans $\N$ et $\xp$ étant irréductible, 
il nous suffit de prouver que la restriction de $\xp$ à $\N^1$ contient le
caractère trivial.
Comme $\N^1$ est un pro-$p$-groupe avec $p\neq\ell$,
le $\flb$-caractère trivial de $\N^1$ est une représentation projective~;
le résultat s'ensuit.
\end{proof}

Soit $\xp$ une $\qlb$-représentation irréductible entière de $\N$.
Si elle n'est pas triviale sur $\N^1$,
alors elle n'ap\-pa\-raît pas dans $\tP\otimes\qlb$
(car $\tP$ est triviale sur $\N^1$ par construction)
et $\rl(\xp)$ ne contient aucun~fac\-teur irréductible trivial sur $\N^1$
d'après le lemme \ref{huysmans}.
Supposons maintenant que $\xp$ soit triviale~sur $\N^1$.
Sa restriction à $\F^\times$ est un multiple d'un caractère $\omega_\xp$.
Si la réduction mod $\ell$ de $\omega_\xp$ n'est pas égale à $\omega_\mp$, 
alors $\mp$ n'apparaît pas dans $\rl(\xp)$
et $\xp$ n'ap\-pa\-raît pas dans $\tP\otimes\qlb$. 
Sinon,
on peut choi\-sir le caractère $\widetilde{\omega}$ de \eqref{girafe}
tel que $\xp_1=\xp \widetilde{\omega}^{-1}$
soit triviale en $\w_\F$.
Alors~:
\begin{itemize}
\item 
la multiplicité de $\mp$ dans $\rl(\xp)$
est égale à la multiplicité de $\mp_1$ dans $\rl(\xp_1)$, 
\item
la multiplicité de $\mp_1$ dans $\rl(\xp_1)$ 
est égale à la multiplicité de $\xp_1$ dans $\tP_1\otimes\qlb$, 
\item
et la multiplicité de $\xp_1$ dans $\tP_1\otimes\qlb$
est égale à la multiplicité de $\xp$ dans $\tP\otimes\qlb$, 
\end{itemize} 
ce dont on déduit le résultat voulu. 
\end{proof}

\begin{rema}
\begin{enumerate}
\item
L'unicité de $\tP$ implique que,
pour tout $\zlb$-ca\-ractère $\chi$ de $\N$ dont la réduction mod $\ell$
est triviale, $\tP\chi$ et $\tP$ sont isomorphes. 
\item
La seconde assertion de la proposition \ref{propVigneras} implique en particulier
qu'une $\qlb$-repré\-sen\-ta\-tion irréductible entière de $\N$ apparaît comme
facteur de $\tP\otimes\qlb$ si et seulement si sa réduction mod $\ell$ 
contient $\mp$.
\end{enumerate}
\end{rema}

\subsection{}
\label{par46}

On reprend les notations du paragraphe \ref{fiduciel}~:
on considère une $\flb$-représentation irréductible $\mp$
de $\N$ triviale sur $\N^1$,
et on note $\P$ son enveloppe projective dans $\rep_{\flb}(\N)$.

\begin{lemm}
\label{decdebase4}
Soit $\mp'$ une $\flb$-représentation irréductible de $\N$
telle que $\Ext^1_\N(\mp,\mp')$ soit non nul. 
Alors $\mp'$ est un sous-quotient irréductible de $\P$.
\end{lemm}

\begin{proof}
Soit $\M$ une extension non triviale de $\mp$ par $\mp'$.
Notons $v$ la surjection de $\M$ sur $\mp$.
Le diagramme suivant~:
\begin{equation*}
\xymatrix{ 0
\ar[r] & \mp' \ar[r] & \M \ar^{v}[r] & \mp \ar[r] & 0 \\ & & &
\P \ar@{.>>}^{g}[ul]  \ar@{->>}_{f}[u] & }
\end{equation*}
(où $f$ est une surjection de $\P$ sur $\mp$)
montre que $\mp'$ apparaît comme sous-quotient de $\P$.
En effet, 
par projectivité de $\P$,
il y a un morphisme~$g$~de $\P$ dans $\M$ tel que $v\circ g=f$,
et il est surjectif,
sans~quoi son image serait incluse dans $\mp'=\Ker(v)$,
qui est l'unique sous-représentation propre maximale de $\M$ car $\M$
est non scindée.
\end{proof}

\begin{lemm}
\label{decdebase5}
Soit $\mp'$ un sous-quotient irréductible de $\P$.
Il existe 
une $\qlb$-représentation~ir\-ré\-duc\-tible entière $\xp$ de $\N$
triviale sur $\N^1$ dont la réduction mod $\ell$ contienne $\mp$ et $\mp'$.
\end{lemm}

\begin{proof}
Il existe des sous-représentations $\Y\subseteq\X$ de~$\P$ telles que 
le quotient $\X/\Y$ soit isomorphe à $\mp'$. 
Soit $\Q$ l'enveloppe
projective de $\mp'$ dans $\rep_\flb(\N)$.
Comme tout morphisme~de~$\Q$ vers $\mp'$ se relève vers $\X\subseteq\P$,
l'espace $\Hom_\N(\Q,\P)$ est non nul.
Notons $\tP$ et $\widetilde{\Q}$ les $\zlb$-représenta\-tions 
projectives de $\N$ relevant $\P$ et $\Q$
données par la proposition \ref{propVigneras}. 
La projectivité de~$\widetilde{\Q}$~en\-traîne que 
$\Hom_\N(\widetilde{\Q},\tP)$ est non nul.
Puis, $\tP$ et $\widetilde{\Q}$ étant libres sur $\zlb$,
on en déduit que~:
\begin{equation*}
\Hom_{\N}\left(\widetilde{\Q}\otimes\qlb , \tP\otimes\qlb \right) \neq \{ 0 \}.
\end{equation*} 
Par conséquent, il existe un facteur irréductible entier ${\xp}$ 
commun à $\widetilde{\Q}\otimes\qlb$ et $\tP\otimes\qlb$.
D'après la~pro\-posi\-tion~\ref{propVigneras}, les représentations
$\mp$ et $\mp'$ apparaissent dans $\rl({\xp})$.
Enfin, le fait que la~repré\-sentation $\xp$ soit triviale sur $\N^1$
découle du lemme \ref{huysmans}.
\end{proof}

\subsection{}
\label{merlinlenchanteur}

Revenons maintenant au paragraphe \ref{floupifloup},
en supposant que $\R$ est le corps $\flb$ et que $i$ est égal à $1$.
On fixe donc un $\flb$-caractère non ramifié $\chi$ de $\G$ tel que l'espace 
$\Ext^1_\N(\xi,\xi^{u}\chi)$ soit~non~nul pour~un $u\in\mathscr{N}_\G(\N^0)$.
Selon les~lemmes~\ref{decdebase4} et~\ref{decdebase5}, 
il y a donc une $\qlb$-représentation~ir\-ré\-duc\-tible~entière $\xp$ de
$\N$~tri\-viale~sur~$\N^1$ dont la~ré\-duc\-tion mod $\ell$
contienne $\xi$ et $\xi'=\xi^{u}\chi$. 
Rappelons que,
pour prouver la proposition \ref{ext1cuspniveau0},
il s'agit de prouver qu'il y a un~$j\in\ZZ$,
qu'on peut prendre égal à $0$ si $\ell$ ne divise
pas $q^n-1$,
tel que~$\chi$ et $\nu^j$ coïncident sur~$\N$.
Compte tenu de \eqref{genN}, l'ordre de la restriction de $\nu$ à $\N$
est égal à celui de $\nu(\w) = q^{-mb}$ dans 
$\overline{\mathbb{F}}{}_\ell^\times$,
et celui-ci vaut $\epsilon(\pi)$
d'après \eqref{defepstor} et la~re\-mar\-que \ref{etiennelousteau}. 

\begin{lemm}
\label{ese}
L'entier $\epsilon(\pi)$ est un multiple de $(s(\pi),\epsilon(\rho))$.
\end{lemm}

\begin{proof}
C'est évident si $\pi$ est supercuspidale, car alors $\rho=\pi$. 
Si $\pi$ n'est pas supercuspidale,
c'est une conséquence de \cite{MSjl} Lemme 3.16,
qui dit que $\epsilon(\pi)$ est égal à $(s(\rho),\epsilon(\rho))$,
et du fait \ref{formek}, qui dit que $s(\rho)=s(\pi)$. 
\end{proof}

Dans les paragraphes \ref{par47} à \ref{genevite}, 
nous allons prouver que l'ordre de $\chi(\w)$
divise $s(\pi)$ et~$\epsilon(\rho)$.~Il 
divisera donc $\epsilon(\pi)$ d'après le lemme \ref{ese}.
Le groupe $\N/\N^0$ étant cyclique,
on en déduira l'exis\-ten\-ce d'un $j\in\ZZ$ tel que 
$\chi$ et $\nu^j$ coïncident sur~$\N$.
Puis nous prouverons au paragraphe \ref{utherpendragon} que,
si $\ell$ ne divise pas $q^n-1$,
les représentations $\pi$ et $\pi\chi$ sont isomorphes. 

\subsection{}
\label{par47}

Fixons un facteur~ir\-ré\-duc\-tible $\W$ de la~res\-triction~de $\xp$ à $\N^0$. 
Celui-ci est l'inflation d'une $\qlb$-représen\-ta\-tion~irréductible $\g$ de 
$\GL_m(\kk_{\D})$,
elle-même sous-quotient irréductible d'une induite~:
\begin{equation}
\label{demarsay}
\tau_1\times\dots\times\tau_r
\end{equation}
où $\tau_i$ est une $\qlb$-représentation irréductible cuspidale de 
$\GL_{m_i}(\kk_{\D})$,
avec $m_1+\dots+m_r=m$. 
Pour chaque $i\in\{1,\dots,r\}$,
notons $\s_i$ la réduction mod $\ell$ de $\tau_i$.
D'après \cite{MSf}~Théo\-rè\-me~2.9,~c'est
une représentation irréductible cuspidale. 
Par hypothèse sur $\xp$,
les représen\-ta\-tions $\s$ et $\s^u$
(la définition de $\s$ étant rappelée au paragraphe \ref{consnxi})
apparaissent dans $\r_\ell(\g)$,
donc dans l'induite~:
\begin{equation}
\label{finot}
\s_1\times\dots\times\s_r,
\end{equation}
ce qui entraîne que $\s$ et $\s^u$ ont le même support supercuspidal.
D'après la proposition \ref{cuspscusp5fini}, 
on en déduit d'une part que~$\s$ et $\s^u$ sont isomorphes, 
\ie que $u\in\N$,
et d'autre part que $\s$ apparaît avec mul\-tiplicité $1$ dans \eqref{finot},
donc \textit{a fortiori} dans $\r_\ell(\g)$.
On tire immédiatement de la proposition \ref{decdebase3}
le corollaire suivant.

\begin{coro}
\label{decdebase3coro}
Il y a un isomorphisme~:
\begin{equation*}
\Ext^1_\G(\pi,\pi\chi) \simeq \Ext^1_\N(\xi,\xi\chi)
\end{equation*}
de $\flb$-espaces vectoriels.
\end{coro}

\subsection{}
\label{revenant}

Revenons à la représentation $\W$,
dont on note $\SS$ le normalisateur dans $\N$.
Par la théorie~de Clifford, 
il existe un prolonge\-ment de $\W$ à $\SS$,
que l'on note encore $\W$, 
dont l'induite à $\N$ soit~iso\-morphe~à~$\xp$.
Comme $\s$ apparaît avec mul\-tiplicité $1$ dans $\r_\ell(\g)$, 
il y a un unique composant~ir\-ré\-ductible $\X$ de $\r_\ell(\W)$
dont la~res\-tric\-tion à $\N^0$ contienne l'inflation $\xi^0$ de $\s$.
Les représen\-ta\-tions
$\xi$ et $\xi'=\xi\chi$ apparaissent donc toutes deux dans 
l'induite de $\X$ à $\N$. 

Notons $\SS'$ le normalisateur de $\X$ dans $\N$,
et fixons un prolongement $\X'$ de $\X$ à $\SS'$. 
L'induite de $\X'$ à $\N$ est~ir\-réductible, et
les composants irréductibles de
l'induite de $\X$ à $\N$ sont de la forme~: 
\begin{equation*}
\label{formecomposantinduite}
\Ind^\N_{\SS'}(\X'\h)
\end{equation*}
où $\h$ est un $\flb$-caractère de $\SS'$ trivial sur $\SS$.
Choisissons $\X'$ de façon que son induite à $\N$ soit~iso\-morphe à $\xi$. 
La restriction de $\xi$ à $\SS'$ étant irréductible,
on en déduit que $\SS'=\N$,
puis~qu'il existe un $\flb$-caractère $\h$ de $\N$ trivial sur $\SS$
tel que $\xi'$ soit isomorphe à $\xi\h$.
Par conséquent, $\chi$ est~trivial sur $\SS$,
\ie que l'ordre de $\chi(\w)$
divise $(\N:\SS)$, 
qui lui-même divise~$s(\pi)$, celui-ci étant par définition
l'in\-dice de $\F^\times\N^0$ dans $\N$.

\subsection{}
\label{genevite}

Plus précisément,
l'ordre de $\chi(\w)$ étant premier à $\ell$,
il divise le plus grand diviseur de $(\N:\SS)$ premier à $\ell$,
que l'on note $e$ dans ce paragraphe.
Nous allons prouver que $e$ divise $\epsilon(\rho)$.

Rappelons que $\r_\ell(\g)$ contient $\s$. 
D'après~\cite{HelmBernstein} Proposition 5.8,
la représentation $\g$ est donc~gé\-né\-rique
(au sens où sa restriction au groupe des matrices triangulaires
uni\-potentes~supé\-rieures
contient le caractère~:
\begin{equation*}
x\mapsto\psi(x_{1,2}+x_{2,3}+\dots+x_{m-1,m})
\end{equation*}
pour n'importe quel caractère non trivial $\psi$ de $\kk_\D$). 
C'est donc l'unique sous-quotient irréductible générique de 
\eqref{demarsay},
où $\g$ apparaît avec multiplicité $1$.
Par conséquent,
le nombre de conjugués~de $\g$ sous $\Gal(\kk_\D/\kk_\F)$,
ou de façon équivalente sous $\N$,
nombre égal à $(\N:\SS)$, 
divise le plus grand multiple commun à $a_1,\dots,a_r$,
où l'entier $a_i$ est le nombre de conjugués de $\tau_i$ sous $\Gal(\kk_\D/\kk_\F)$.
Il suffit donc de prouver que $e_i$,
le plus grand diviseur de $a_i$ premier à $\ell$,
divise $\epsilon(\rho)$.
Appliquant la proposition \ref{cuspscusp5fini},
et compte tenu du fait \ref{bitcoink},
on~a~: 
\begin{equation*}
\scusp(\s) = \a+\dots+\a
\end{equation*} 
où $\a$ est une représentation supercuspidale de $\GL_{m/k}(\D)$.
D'après \cite{MSjl} Lemmes 3.10, 
3.13,~l'en\-tier $e_i$ est égal soit à $1$, soit à l'ordre de
$q^{bm_i}$ mod $\ell$. 
Par ailleurs,
la représentation $\s$ apparaissant dans l'induite \eqref{finot},
on déduit de l'unicité du support supercuspidal que
le support supercuspidal de $\s_i$ est une som\-me de copies de $\a$.
Par conséquent,
$m_i$ est un multiple de $m/k$,
donc $e_i$ divise l'ordre de $q^{bm/k}$ mod $\ell$,
qui est égal à $\epsilon(\rho)$. 

\subsection{}
\label{utherpendragon}

Il ne nous reste plus qu'à prouver que, 
si $\ell$ ne divise pas $q^n-1$,
alors $\chi$ est trivial sur $\N$.
On suppose donc que 
$\ell$ ne divise pas $q^n-1$,
\ie que $\epsilon(\pi)$ ne divise pas $s(\pi)$.
Il suit alors~de \cite{MSjl} Remarque 6.1 que $\s$, 
ou de façon équiva\-lente~$\pi$ d'après~le~fait~\ref{bitcoin0},
est supercuspidale, et~on déduit
de \cite{Vigb} III.2.9 que $\g$ est cuspidale.
D'après \cite{MSt} Proposition 3.1,
l'induite compacte~$\mu$~de la représentation $\W$
(du groupe $\SS$ du paragraphe \ref{revenant})
à $\G$ 
est une $\qlb$-représentation irréductible cuspidale 
dont~la réduction mod $\ell$ contient $\pi$ et $\pi\chi$. 
Mais le corollaire \ref{coromal2q3} dit que $\r_\ell(\mu)$
est~ir\-ré\-ductible,
donc $\pi\chi$ est isomorphe à $\pi$.

\section{Le cas cuspidal de niveau quelconque}
\label{sec5}

Dans cette section, 
on étudie le cas des représentations cuspidales quelconques de $\G$. 

\subsection{}
\label{supertypes}

Dans ce paragraphe, $\R$ est un corps algébriquement clos de
caractéristique~dif\-férente de $p$.

Soit $\pi$ une $\R$-représentation irréductible cuspidale de $\G$, 
et soit $\Omega=\Omega(\pi)$ sa classe~iner\-tiel\-le.
D'après le théorème \ref{decbvss}, 
la catégorie $\Rep_\R(\G,\Omega)$, formée des~re\-présen\-tations 
dont les~sous-quo\-tients irré\-ductibles ont leur support supercuspidal
dans $\Omega$, 
est un facteur direct~indé\-compo\-sable de $\Rep_\R(\G)$. 

Fixons une paire $(\J^0,\bl)$ comme au paragraphe \ref{flapflip},
dont nous utiliserons les notations.~Rappe\-lons
qu'on a une extension finie $\E$ de $\F$, 
une $\E$-algèbre $\B$ qu'on identifie à $\M_r(\C)$ via un isomorphisme 
\eqref{melon} fixé~une~fois pour toutes,
que $\J^0$ a un unique pro-$p$-sous-groupe distingué maximal $\J^1$
et que la restriction de $\bl$ à $\J^1$ est un multiple d'une 
représentation irréductible $\bn$.
Dans~tou\-te cette section, on pose~:
\begin{equation*}
\G_0=\B^\times\simeq\GL_r(\C).
\end{equation*}
D'après~\cite{SEns} Proposition 10.2,
l'induite~com\-pac\-te $\ind_{\J^0}^\G(\bl)$~ap\-par\-tient à
$\Rep_\R(\G,\Omega)$.

\subsection{}
\label{sansnumero}

Nous allons maintenant décrire certains résultats dus à Chinello 
\cite{gianmarco3}
que nous utiliserons pour nous 
ramener à un bloc~de ni\-veau
$0$ de $\Rep_{\flb}(\G_0^\times)$.
Notons $\Rep_{\flb}(\G,\bn)$ la sous-catégorie pleine formée des
représentations de $\Rep_{\flb}(\G)$ engendrées par leur composante
$\bn$-isotopique. 
La caté\-gorie $\Rep_{\flb}(\G,\Omega)$ introduite au paragraphe précédent
en est un facteur direct. 

Notons 
$\Hh_{{\flb}}(\G,\bn)$ la~${\flb}$-algèbre des fonctions à support compact 
$\Psi : \G \to \End_{\flb}(\n)$
telles que $\Psi (xgy)=\n(x)\circ\Psi (g)\circ\n(y)$ quels que soient $x,y\in\J^1$, 
munie du produit~:
\begin{equation*}
\Psi*\Psi' : g \mapsto \sum\limits_{h} \Psi (h)\Psi '(h^{-1}g)
\end{equation*}
où $h$ décrit un système de représentants de $\G/\J^1$.
Notant $\Mod(\Hh_\flb(\G,\bn))$ la~caté\-gorie des~mo\-du\-les à droite sur
$\Hh_\flb(\G,\bn)$, on a un foncteur~: 
\begin{eqnarray}
\label{ligneta}
\MM_{\bn} : \Rep_\flb(\G,\bn) &\to& \Mod(\Hh_\flb(\G,\bn)) \\
\notag
\V &\mapsto& \Hom_\G\left(\ind^\G_{\J^1} (\bn), \V\right)
\end{eqnarray}
où l'action (à gauche)
de $\Psi\in\Hh_\flb(\G,\n)$ sur $f\in\ind^\G_{\J^1}(\n)$
est définie par la formule~:
\begin{equation*}
\Psi * f : g \mapsto \sum\limits_{h} \Psi(h) f(h^{-1}g)
\end{equation*}
et l'action (à droite) de $\Psi$ sur $\h\in\MM_{\bn}(\V)$
est donnée par $\h\cdot\Psi:f\mapsto\h(\Psi*f)$.
On a le résultat important suivant. 

\begin{theo}[\cite{gianmarco3} Theorem 5.10]
Le foncteur $\MM_\n$ est une équivalence de catégories. 
\end{theo}

Observons que,
dans le cas particulier où $\pi$ est de niveau $0$,
$\G_0$ est égal à $\G$ et $\eta$ est le~caractè\-re~trivial du
groupe $\J^1=\N^1=1+\Mat_m(\p_\D)$.
Le foncteur \eqref{ligneta} définit alors une équivalence~de~la ca\-tégorie 
$\Rep_\flb(\G,\N^1)$ des $\flb$-représentations de $\G$
qui sont de niveau $0$, 
\ie qui sont~en\-gen\-drées par leurs~vec\-teurs $\N^1$-invariants,
vers la catégorie $\Mod(\Hh_\flb(\G,\N^1))$
des modules à~droi\-te sur la $\flb$-algèbre des fonctions de $\G$
dans $\flb$ à support compact et bi-invariantes par $\N^1$.

\subsection{}
\label{foncGG}

Notons $\K^0$ le sous-groupe compact maximal $\GL_r(\Oo_\C)$ de $\G_0$
et $\K^1$ son pro-$p$-sous-groupe~dis\-tingué maximal.
Comme au paragraphe~\ref{sansnumero}, on définit
la sous-ca\-té\-gorie pleine $\Rep_\flb(\G_0,\K^1)$~de $\Rep_\flb(\G_0)$
formée des $\flb$-représentations engendrées par leurs~vec\-teurs
$\K^1$-invariants, 
la $\flb$-al\-gè\-bre $\Hh_{\flb}(\G_0,\K^1)$ des fonctions de $\G_0$ dans $\flb$
à support compact et bi-invariantes par $\K^1$, et le fonc\-teur~:
\begin{eqnarray*}
\MM_{1} : \Rep_\flb(\G_0,\K^1) &\to& \Mod(\Hh_\flb(\G_0,\K^1)) \\
\V &\mapsto&
\Hom_{\G_0}\left(\ind^{\G_0}_{\K^1} (1), \V\right)
\end{eqnarray*} 
qui est une équivalence de catégories selon
\cite{gianmarco2} théorème 3.2. 
Dans \cite{gianmarco3},
Chinello construit un~iso\-morphisme de $\flb$-algèbres~:
\begin{equation*}
\label{isotheta}
\t : \Hh_{\flb}(\G_0,\K^1) \to \Hh_{\flb}(\G,\bn)
\end{equation*}
dépendant du choix de la représentation $\k$ prolongeant $\n$
au groupe $\J^0$ fixée au paragraphe \ref{flapflip}, 
et ayant la~pro\-priété suivante (voir \cite{gianmarco3} Section 3,
et en particulier \cite{gianmarco3} Theorem 3.43). 

\begin{enonce}{Fait}
\label{chourico}
Soit $\w_\C$ une uniformisante de $\C$,
soit un entier $i\in\{0,\dots,r\}$ et soit la~matrice~dia\-gonale~:
\begin{equation*}
b={\rm diag}(1,\dots,1,\w_\C,\dots,\w_\C) \in \G_0
\end{equation*}
où $1$ apparaît $i$ fois. 
Si $\Psi_0\in\Hh_{\flb}(\G_0,\K^1)$ est à support dans $\K^1b\K^1$,
son image $\t(\Psi_0)$ est~à~sup\-port dans $\J^1b\J^1$.
\end{enonce}

Cet isomorphisme d'algèbres $\t$
définit une équivalence de catégories~:
\begin{equation*}
\theta^* : \Mod(\Hh_{\flb}(\G,\bn)) \to \Mod(\Hh_{\flb}(\G_0,\K^1)).
\end{equation*}
Composant les foncteurs $\MM_\n$ et $\t^*$ avec un
quasi-inverse de $\MM_1$,
Chinello obtient une équivalen\-ce~de~catégories~:
\begin{equation*}
\GG : \Rep_{\flb}(\G,\bn) \to \Rep_{\flb}(\G_0,\K^1).
\end{equation*}
Le diagramme commutatif suivant~:
\begin{equation*}
\xymatrix{\Rep_{\flb}(\G,\bn) \ar[r]^{\GG}  \ar[d]_{\MM_{\bn}} & \Rep_{\flb}(\G_0,\K^1)
\ar[d]^{\MM_1}\\ \Mod(\Hh_{\flb}(\G,\bn)) \ar[r]_{\theta^*} &
\Mod(\Hh_{\flb}(\G_0,\K^1))}
\end{equation*} 
résume la situation.

\subsection{}

Considérons maintenant la représentation $\bs$ de $\J^0$ triviale sur $\J^1$
du paragraphe \ref{flapflip},
telle que $\l^0$ soit isomorphe à $\k\otimes\bs$.
Elle~dé\-fi\-nit, par restriction, une représentation $\bs_0$ de
$\J^0\cap\G_0\simeq\K^0$ triviale sur $\J^1\cap\G_0\simeq\K^1$,
inflation de la représentation cuspidale $\s$ du groupe~:
\begin{equation}
\label{MmedeNegrepelisse}
\J^0/\J^1\simeq\K^0/\K^1\simeq\GL_r(\ll).
\end{equation}

Notons $\K$ le normalisateur de~la classe d'isomorphisme de $\bs_0$
dans $\G_0$,
et fixons un prolongement $\xi^{\phantom{0}}_0$ de $\bs_0$ à $\K$.
D'après \cite{MSt} Proposition 3.1,
l'induite compacte de $\xi_0$ à $\G_0$,~no\-tée $\pi_0$,
est une $\flb$-re\-pré\-sen\-tation
cuspi\-da\-le de niveau $0$ de $\G_0$. 
Par construction de $\pi_0$, on a le résultat suivant.

\begin{lemm}
\label{quicklemma}
On a $q(\pi)=q(\pi_0)$.
\end{lemm}

Comme au paragraphe \ref{flapflip},
la représentation cuspidale
$\pi_0$ définit un facteur direct~indé\-com\-posable
$\Rep_\flb(\G_0,\Omega_0)$ de $\Rep_\flb(\G_0,\K^1)$,
où $\Omega_0$ est la classe inertielle de $\pi_0$,
formé des re\-pré\-sen\-ta\-tions 
dont les sous-quotients ir\-ré\-ductibles
ont leur support supercuspidal
dans $\Omega_0$.
D'après \cite{gianmarco3} Theorem 5.15, on a le résultat suivant. 

\begin{prop}
\label{urgence}
Le foncteur $\GG$ induit une équivalence entre $\Rep_\flb(\G,\Omega)$
et $\Rep_\flb(\G_0,\Omega_0)$.
\end{prop}

\subsection{}
\label{heronseiche}

L'objet de ce paragraphe est de prouver le résultat suivant.

\begin{lemm}
\label{lemmeheronseiche}
La représentation $\GG(\pi)$ est cuspidale.
\end{lemm}

\begin{proof}
Notons $\V$ la $\flb$-représentation $\GG(\pi)$,
et considérons l'espace $\V^{\K^1}$
des vecteurs de $\V$ invariants sous $\K^1$, 
canoniquement isomorphe~à~:
\begin{equation*}
\label{lrouge}
\MM_1(\V) = \Hom_{\G_0}\left(\ind^{\G_0}_{\K^1} (1), \V\right).
\end{equation*}
C'est un module à~droite sur $\Hh_\flb(\G_0,\K^1)$. 
Si on le restreint à la sous-algèbre $\Hh_\flb(\K^0,\K^1)$
des fonctions à support dans $\K^0$
et si on identifie celle-ci à l'algèbre de groupe de $\K^0/\K^1$,
on obtient la représentation~na\-tu\-relle de $\K^0/\K^1$
sur $\V^{\K^1}$,
que l'on voit \textit{via} \eqref{MmedeNegrepelisse} comme une représentation
du groupe $\GB=\GL_r(\kk_\C)$, que l'on note $\Vv$.

Considérons maintenant l'espace $\Hom_{\J^1}(\n,\pi)$,
canoniquement isomorphe~à~:
\begin{equation*}
\label{lbleu}
\MM_\n(\pi) = \Hom_{\G}\left(\ind^{\G}_{\J^1} (\n), \pi\right)
\end{equation*}
par réciprocité de Frobenius. 
D'après la propriété (4.c) du paragraphe \ref{flapflip},
cet espace définit une représentation cuspidale~de $\GB$ qui,
en~ver\-tu du diagramme commutatif \cite{gianmarco3} (12),
est isomorphe à $\Vv$.
Il s'ensuit que $\Vv$~est~cus\-pidale. 

Supposons maintenant que $\V$ ne soit pas cuspidale.
Il existe donc un sous-groupe parabolique standard propre
$\P_0=\M_0\U_0$ de $\G_0$ et une représentation irréductible 
$\W$ de $\M_0$ tels que $\V$ soit~un sous-quotient 
de l'induite parabolique de $\W$ à $\G_0$.
Notant respectivement $\PB$ et $\MB$ les images~de $\P_0\cap\K^0$ 
et $\M_0\cap\K^0$ dans $\GB$, 
ainsi que $\Ind_{\P_0}^{\G_0}$ et $\Ind^{\GB}_{\PB}$ les foncteurs 
d'induction parabolique~cor\-respondant à $(\P_0,\M_0)$ et 
$(\PB,\MB)$, on a un isomorphisme~:
\begin{equation}
\label{lvert}
\left(\Ind_{\P_0}^{\G_0}(\W)\right)^{\K^1} \simeq
\Ind^{\GB}_{\PB}\left(\W^{\M_0\cap\K^1}\right),
\end{equation}
de représentations de $\GB$, 
l'espace $\W^{\M_0\cap\K^1}$ des vecteurs de $\W$ invariants par 
$\M_0\cap\K^1$ étant~con\-sidéré comme une représentation de 
$(\M_0\cap\K^0)/(\M_0\cap\K^1)\simeq \MB$.
(L'isomorphisme \eqref{lvert} se déduit de l'identité $\G_0=\P_0\K^0$~;
c'est aussi un cas particulier de \cite{SEns} Proposition 5.6.)
Enfin, le sous-groupe $\K^1$ étant un pro-$p$-groupe, 
le foncteur des $\K^1$-invariants est exact. 
La représentation $\Vv$ est donc un sous-quotient de \eqref{lvert},
ce qui contredit le fait qu'elle est cuspidale. 
\end{proof}

Les représentations $\GG(\pi)$ et $\pi_0$ sont donc toutes deux cuspidales,
et leurs supports supercuspidaux sont inertiellements équivalents.
Raisonnant comme dans la preuve du lemme \ref{decdebase25},
on en déduit qu'elles sont tordues l'une de l'autre par
un caractère non ramifié de $\G_0$.
Nous pouvons donc choisir $\xi_0^{}$
de sorte que $\GG(\pi)$ et $\pi_0$ soient isomorphes,
ce que nous ferons dorénavant.

\subsection{}

Nous allons montrer que $\GG$ se comporte bien vis-à-vis de la torsion par un
caractère non~rami\-fié,
afin d'utiliser les résultats de la section précédente.

\begin{lemm}
\label{adieulescons}
Soit $\chi$ un $\flb$-caractère non ramifié de $\G$,
et soit $\chi_0$ la restriction de $\chi$ à $\G_0$.
Alors $\GG(\pi\chi)$ est isomorphe à $\pi_0\chi_0$.
\end{lemm}

Posons $\Hh=\Hh_\flb(\G,\n)$. 
Si $f\in\ind^\G_{\J^1}(\n)$ et si $\chi$ est un caractère non ramifié
de $\G$,~on note $\chi f$ la fonction $g\mapsto\chi(g)f(g)$ de
$\ind^\G_{\J^1}(\n)$.
Si $\Psi\in\Hh$,
on note ${\chi}\Psi$ la fonction $f\mapsto\chi(g)\Psi(g)$ de $\Hh$.

\begin{rema}
\label{chak}
On observera que,
si $f\in\ind^\G_{\J^1}(\n)$ a pour support $\J^1 g$ pour un $g\in\G$,
alors $\chi f$ est simplement égale à $\chi(g)f$.
De façon analogue,
si $\Psi\in\Hh$ a pour support $\J^1 g\J^1$ pour un $g\in\G$,
alors $\chi\Psi$ est simplement égale à $\chi(g)\Psi$.
\end{rema}

Il est commode d'introduire la définition suivante. 

\begin{defi}
Si $\M$ est un $\Hh$-module à droite 
et si $\chi$ est un caractère non ramifié de $\G$,~on
note $\M\chi$ le $\Hh$-module à droite
d'espace sous-jacent $\M$,
muni de l'action de $\Hh$ donnée par~:
\begin{equation*}
(v,\Psi) \mapsto v \cdot ({\chi^{-1}}\Psi),
\quad
v\in\M,
\quad
\Psi\in\Hh,
\end{equation*}
où $\cdot$ désigne l'action de $\Hh$ sur $\M$.
\end{defi}

Le lemme suivant justifie la définition précédente. 

\begin{lemm}
\label{Lemme 4.1.4}
Soit $\pi$ une représentation dans $\Rep_{\flb}(\G,\n)$,
et soit $\chi$ un caractère non ramifié de $\G$.
Les $\Hh$-modules $\MM_{\bn}(\pi \chi)$ et $\MM_{\bn}(\pi) \chi$
sont isomorphes.  
\end{lemm}

\begin{proof}
Pour tout vecteur $w$ dans l'espace de $\n$,
on note $\i_{w}$ l'élément de $\ind^\G_{\J^1}(\n)$~de support $\J^1$
prenant la valeur $w$ en $1$.
Par réciprocité de Frobenius,
il correspond à tout~morphis\-me
$\h\in\Hom_\G(\ind^\G_{\J^1}(\n),\pi\chi)$ le morphisme~:
\begin{equation*}
w \mapsto \h(\i_{w})
\end{equation*}
de $\Hom_{\J^1}(\n,\pi\chi)$.
Réciproquement, à tout $\psi\in\Hom_{\J^1}(\n,\pi\chi)$ correspond
le morphisme~:
\begin{equation*}
f \mapsto \sum\limits_{g} \chi(g)^{-1} \pi(g)^{-1} \psi(\i_{f(g)})
\end{equation*}
de $\Hom_\G(\ind^\G_{\J^1}(\n),\pi\chi)$,
où $g$ décrit un système de représentants de $\J^1\backslash\G$.
Les $\flb$-espaces vectoriels 
$\Hom_{\J^1}(\n,\pi\chi)$ et $\Hom_{\J^1}(\n,\pi)$
étant égaux, on peut associer à $\h$ le morphisme~:
\begin{equation*}
\h^* : f \mapsto \sum\limits_{g} \pi(g)^{-1}\h(\i_{f(g)})
\end{equation*}
de $\Hom_\G(\ind^\G_{\J^1}(\n),\pi)$.
On définit de cette façon un isomorphisme d'espaces vectoriels
$\h\mapsto\h^*$~de $\MM_{\bn}(\pi \chi)$ vers $\MM_{\bn}(\pi) \chi$, 
et~nous allons vérifier que c'est un isomorphisme de $\Hh$-modules.

Soit $\Psi\in\Hh$.
Il s'agit de prouver que, pour tout $w$ dans l'espace de $\n$,
on a~:
\begin{equation}
\label{LEF}
(\h \cdot\Psi)(\i_w) = (\h^*\cdot\chi^{-1}\Psi)(\i_w).
\end{equation}
Posons $f=\Psi*\i_{w}\in\ind^\G_{\J^1}(\n)$,
qui n'est autre que la fonction $g\mapsto\Psi(g)w$.
On a~:
\begin{equation*}
{\chi^{-1}}\Psi(\i_{w}) = (\chi^{-1}\Psi)*\i_{w} = \chi^{-1}f.
\end{equation*}
Le membre de droite de \eqref{LEF} donne~:
\begin{equation*}
\h^*(\chi^{-1}f)
= \sum\limits_{g} \pi(g)^{-1} \h(\i_{\chi(g)^{-1} f(g)})
= \sum\limits_{g} \chi(g)^{-1}\pi(g)^{-1} \h(\i_{f(g)})
= \h(f)
\end{equation*}
ce qui prouve le résultat escompté. 
\end{proof}

Posons $\Hh_0=\Hh_\flb(\G_0,\K^1)$.
On a une propriété analogue au lemme \ref{Lemme 4.1.4}
pour les $\Hh_0$-modules.

Prouvons maintenant le lemme \ref{adieulescons}.
Raisonnant comme au paragraphe \ref{heronseiche},
le fait que $\GG(\pi\chi)$ soit une représentation cuspidale
de $\Rep_\flb(\G_0,\Omega_0)$ entraîne 
qu'il existe un caractère non ramifié 
$\mu_0$ de $\G_0$
tel que $\GG(\pi\chi)$ soit isomorphe à $\pi_0\mu_0$.
Appliquant $\MM_1$,
et~comp\-te tenu du lemme~\ref{Lemme 4.1.4},~on 
obtient un isomorphisme 
$\t^*(\MM_\n(\pi)\chi) \simeq \t^*(\MM_\n(\pi))\mu_0$
de $\Hh_0$-modules.  
Par définition, cela signi\-fie que~:
\begin{equation*}
\t({\mu_0}\Psi_0)={\chi}\t(\Psi_0) 
\end{equation*}
pour tout $\Psi_0\in\Hh_0$.

Supposons maintenant que $\Psi_0$ soit
la fonction caracté\-ristique 
de la double-classe $\K^1b\K^1$ avec
$b={\rm diag}(1,\dots,1,\w_\C)\in\G_0$.
D'après la remarque \ref{chak}, 
on a ${\mu_0}\Psi_0=\mu_0(b)\Psi_0$.
Ensuite,
d'après~le fait \ref{chourico},
la fonction $\Psi=\t(\Psi_0)$ a pour support $\J^1b\J^1$.
D'après la remarque \ref{chak} à nouveau, 
on a donc $\chi\Psi=\chi(b)\Psi$.
On en déduit que $\mu_0(b)=\chi(b)$.

Le caractère $\chi$ est de la forme
$\a\circ{\rm Nrd}$ où $\a$ est un caractère non ramifié
de $\F^\times$ et ${\rm Nrd}$~dési\-gne la norme réduite de
$\Mat_m(\D)$ sur $\F$.
De façon analogue, 
le caractère $\mu_0$ est de la forme
$\a_0\circ{\rm Nrd}_{\B}$ où~$\a_0$ est un caractère non ramifié
de $\E^\times$ et ${\rm Nrd}_\B$ est la norme réduite de
$\B$ sur~$\E$.
On a~:
\begin{eqnarray*}
\chi(b) &=& \a\circ{\rm Nrd}(b) \\ 
&=& \a\circ\N_{\E/\F}\circ{\rm Nrd}_\B(b) \\ 
&=& \a\circ\N_{\E/\F}({\rm Nrd}_\C(\w_\C)) 
\end{eqnarray*}
et $\w_\E={\rm Nrd}_\C(\w_\C)$
est une uniformisante de $\E$.
Un calcul analogue donne $\mu_0(b)=\a_0(\w_\E)$.
On en déduit que $\a_0=\a\circ\N_{\E/\F}$,
donc que $\mu_0$ est égal à $\chi_0$,
la restriction de $\chi$ à $\G_0$.

\subsection{}

On utilise notre travail ci-dessus
pour obtenir la proposition suivante qui généralise la proposition
\ref{ext1cuspniveau0}.

\begin{prop}
\label{ext1cuspniveaunonnul}
Soient $\pi$, $\pi'$ des $\flb$-représentations cuspidales de $\G$.
Suppo\-sons~que l'es\-pa\-ce $\Ext^1_\G(\pi,\pi')$ soit non nul.
Alors $\pi'\in\B(\pi)$.
\end{prop}

\begin{proof} 
Comme les représentations $\pi'$ et $\pi$ sont dans le même
bloc de $\Rep_\flb(\G)$,
on en déduit que $\pi'$ est dans $\Rep_\flb(\G,\Omega)$.
On peut donc appliquer $\GG$,
ce qui donne~:
$$\Ext^1_{\G_0}(\GG(\pi),\GG(\pi')) \neq \left \{ 0 \right \}$$
dans $\Rep_{\flb}(\G_0,\Omega_0)$,
et $\GG(\pi)$ est isomorphe à $\pi_0$.
D'après la proposition \ref{ext1cuspniveau0}, ceci entraîne que~:
\begin{enumerate}
\item
il existe un entier $j \in \ZZ$ tel que $\GG(\pi')$ soit isomorphe à
$\pi_0^{\phantom{j}}\nu_0^j$,
\item
si $\ell$ ne divise pas $q(\pi_0)-1$, 
alors $\GG(\pi')$ est isomorphe à $\pi_0$,
\end{enumerate} 
où $\nu_0$ est le caractère non ramifié
``valeur absolue de la norme réduite'' de $\G_0$.
D'après le lemme \ref{quicklemma}, on a $q(\pi_0)=q(\pi)$.
D'après le lemme \ref{adieulescons},
et comme la restriction de $\nu$ à $\G_0$ est égale à $\nu_0$, 
on a donc~: 
\begin{enumerate}
\item
il existe un entier $j \in \ZZ$ tel que $\GG(\pi')$ soit isomorphe à
$\GG(\pi\nu^j)$,
\item
si $\ell$ ne divise pas $q(\pi)-1$, 
alors $\GG(\pi')$ est isomorphe à $\GG(\pi)$.
\end{enumerate} 
Le foncteur $\GG$ étant une équivalence de catégories, on trouve le résultat 
annoncé. 
\end{proof}

\section{Le cas général}
\label{sec6}

Dans cette section,
on décompose la catégorie $\rep_{\flb}(\G)$ en blocs.

\subsection{}

On considère à présent une $\flb$-représentation irréductible quelconque $\pi$ 
de $\G$.
D'après la~dé\-fi\-nition \ref{Notation C(pi)},
il lui correspond un ensemble $\B(\pi)$.
Rappelons que, 
si le support supercuspidal de $\pi$ est $\rho_1+\dots+\rho_r$,
alors $\B(\pi)$ est l'ensemble des $\flb$-représentations irrédu\-ctibles de $\G$
dont le support supercuspidal est de la forme~:
\begin{equation}
\label{aprouver}
\rho_1\nu^{j_1}+\dots+\rho_r\nu^{j_r},
\quad
j_1,\dots,j_r\in\ZZ,
\end{equation} 
où, pour chaque $k=1,\dots,r$, l'entier $j_k$ est nul 
si $\ell$ ne divise pas $q(\rho_k)-1$.
Nous allons prouver le résultat suivant,
qui~gé\-né\-ralise la proposition \ref{ext1cuspniveaunonnul}.  
Pour cela, nous nous inspi\-rons de la preuve de
\cite{EH} Theorem 3.2.13.

\begin{prop}
\label{Prop 3.2.1}
Soient $\pi$, $\pi'$ des $\flb$-représentations irréductibles de $\G$.
Supposons que~l'es\-pa\-ce 
$\Ext^1_\G(\pi',\pi)$ soit non nul.
Alors $\pi'\in\B(\pi)$. 
\end{prop}

\begin{proof}
Si $\pi$ et $\pi'$ sont cuspidales,
le résultat est donné par la proposition
\ref{ext1cuspniveaunonnul}.~Sup\-po\-sons
maintenant que $\pi'$ soit cuspidale mais pas $\pi$.
Il y a une représentation irréductible~cus\-pi\-dale $\tau$ d'un sous-groupe de
Levi standard propre $\M$ de $\G$ et un sous-groupe parabolique standard
$\P$ de $\G$
de facteur de Levi $\M$ tels que $\pi$ se plonge dans $\ip^\G_\P(\tau)$.
On a une suite exacte courte~:
\begin{equation*}
0 \to \pi \to \ip^\G_\P(\tau) \to \d \to 0
\end{equation*}
qui définit $\d$,
et on en déduit la suite exacte~:
\begin{equation}
\label{seq}
\Hom_\G(\pi',\d) \to \Ext^1_\G(\pi',\pi) \to \Ext^1_\G(\pi',\ip^\G_\P(\tau)).
\end{equation}
Notant $\U$ le radical unipotent de $\P$
et $\pi'_\U$ le module de Jacquet de $\pi'$
relativement au triplet~pa\-ra\-bolique $(\P,\M,\U)$,
on a par adjonction un isomorphisme de $\flb$-espaces vectoriels~:
\begin{equation}
\label{isoext1}
\Ext^1_\G(\pi',\ip^\G_\P(\tau)) \simeq \Ext^1_\M(\pi'_\U,\tau)
\end{equation}
et le membre de droite est nul car $\pi'$ est cuspidale.
Par conséquent, $\Hom_\G(\pi',\d)$ est~non~nul~:
on~en déduit que $\pi'$ a le même support supercuspidal que $\pi$,
donc que $\pi'\in\B(\pi)$. 
Si $\pi$ est~cus\-pi\-dale~mais pas $\pi'$,
on se ramène au cas précédent par passage aux contragrédientes
\textit{via} l'isomor\-phis\-me de $\flb$-espaces vectoriels~:
\begin{equation*}
\Ext^1_\G(\pi',\pi) \simeq \Ext^1_\G(\pi^\vee,\pi'^\vee).
\end{equation*} 
Supposons enfin que ni $\pi$ ni $\pi'$ ne soient cuspidales,
et formons à nouveau la suite~exac\-te \eqref{seq}.
Si $\Hom_\G(\pi',\d)$ est~non~nul, on a $\pi'\in\B(\pi)$.
Sinon, 
$\Ext^1_\G(\pi',\ip^\G_\P(\tau))$ est non nul
et un argument~de dévissage standard implique,
compte tenu de \eqref{isoext1}, 
qu'il y a un sous-quotient irréductible~$\a$~de $\pi'_\U$
tel que $\Ext^1_\M(\a,\tau)$ soit non nul.
Identifions $\M$ à un produit
$\GL_{m_1}(\D)\times\dots\times\GL_{m_l}(\D)$
pour des entiers $m_1,\dots,m_l\>1$ de somme $m$,
et écrivons~:
\begin{equation*}
\Ext^1_\M(\a,\tau) \simeq \Ext^1_{\GL_{m_1}(\D)}(\a_1,\tau_1) \otimes\dots\otimes
\Ext^1_{\GL_{m_l}(\D)}(\a_l,\tau_l) 
\end{equation*}
où $\a\simeq\a_1\otimes\dots\otimes\a_l$ et
$\tau\simeq\tau_1\otimes\dots\otimes\tau_l$.
Pour chaque $i\in\{1,\dots,l\}$,
l'espace $\Ext^1_{\GL_{m_i}(\D)}(\a_i,\tau_i)$ est~non nul.
Comme la représentation $\tau_i$ est cuspidale,
un des cas précédemmment traités~impli\-que que 
$\a_i\in\B(\tau_i)$. 
Par ailleurs, d'après le lemme géométrique (\cite{Datnu} 2.8), on a~:
\begin{equation*}
\sum\limits_{i=1}^{l} \scusp(\a_i) = \scusp(\pi').
\end{equation*}
Il s'ensuit que $\pi'\in\B(\pi)$. 
Ceci met fin à la démonstration de la proposition \ref{Prop 3.2.1}.
\end{proof}

\subsection{}

Notons $\mathscr{B}=\mathscr{B}(\G)$
l'ensemble des $\B(\pi)$ quand $\pi$
parcourt les représentations irréductibles~du groupe $\G$.
\'Enonçons le premier résultat principal de l'article.

\begin{theo}
\label{Prop 3.2.4}
On a une décomposition en blocs~:
\begin{equation}
\label{decb}
\rep_\flb (\G) = \bigoplus\limits_{\B} \rep_\flb (\G,\B)
\end{equation}
où $\B$ décrit les éléments de $\mathscr{B}$,
et où $\rep_\flb (\G,\B)$ est la sous-catégorie pleine
de $\rep_{\flb}(\G)$ formée des représentations dont tous
les sous-quotients irréductibles sont dans $\B$. 
En d'autres termes, 
toute $\flb$-représentation de longueur finie $\V$ de $\G$
admet une unique décomposition~:
\begin{equation}
\label{decbv}
\V=\bigoplus\limits_{\B} \V(\B)
\end{equation}
où $\V(\B)$ désigne la plus grande sous-représentation de $\V$
dont tous les sous-quotients irréductibles sont dans $\B$.
\end{theo}

\begin{proof}
On raisonne par récurrence sur la longueur de $\V$,
le cas de longueur $1$ étant immédiat puisque les
$\B\in\mathscr{B}$ sont disjoints grâce à l'unicité
du support supercuspidal. 

Soit $\V$ une représentation de longueur finie $\>2$ de $\G$,
soit $\pi$ une sous-représentation~irréduc\-tible de $\V$ et
posons $\B=\B(\pi)$. 
La proposition \ref{Prop 3.2.1} assure que,
pour toute représentation $\pi'\in\B$ et toute représentation
$\s\in\Irr(\G)-\B$,
l'espace d'extension $\Ext^1_\G(\pi',\s)$ est nul. 
Par conséquent,
d'après le lemme \ref{vanessa},
la représentation~$\V$~se~dé\-com\-pose en
$\V=\V(\B)\oplus\W$
où $\W$ est la plus~gran\-de~sous-représentation de $\V$ dont
les sous-quo\-tients irréductibles sont hors de $\B$.
Comme $\V(\B)$ est non nul,
la longueur de $\W$ est strictement moindre que celle de $\V$.
On peut donc lui appliquer l'hypothèse de récurrence. 

Enfin, le fait que les facteurs $\rep_{\flb}(\G,\B)$ sont indécomposables est
une conséquence de la~pro\-position \ref{blocindec} et du lemme 
\ref{lemdep}.
\end{proof}

\begin{rema}
\label{Bsingleton6}
Si $\G=\GL_n(\F)$,
la remarque \ref{Bsingleton} montre que des représentations
irréduc\-ti\-bles sont dans le même bloc si et seulement si elles ont le
même support supercuspidal.
\end{rema}

On déduit du théorème \ref{Prop 3.2.4} le résultat suivant, 
qui généralise la proposition \ref{Prop 3.2.1}.

\begin{coro}
\label{Prop 3.2.1 cor}
Soient $\pi$ et $\pi'$ des $\flb$-représentations irréductibles de $\G$.
Supposons qu'il y ait un entier $i \geq 0$ tel que
$\Ext^i_\G(\pi',\pi)$ soit non nul.
Alors $\pi'\in\B(\pi)$. 
\end{coro}

\subsection{}

Comme annoncé à la fin de la section \ref{sec3},
nous terminons cette section par le résultat suivant.

\begin{prop}
\label{blocindecfinal}
Soient $\pi$, $\pi'$ des $\flb$-représentations~ir\-réductibles de $\G$. 
Les assertions~sui\-vantes sont équivalentes~:
\begin{enumerate} 
\item
Il existe une $\qlb$-représentation irréductible entière de $\G$ 
dont la réduction mod $\ell$ contienne à la fois $\pi$ et $\pi'$.
\item
Les ensembles $\B(\pi)$ et $\B(\pi')$ sont égaux. 
\item 
Les représentations $\pi$ et $\pi'$ sont dépendantes. 
\end{enumerate}
\end{prop}

\begin{proof}
Les première et seconde assertions sont équivalentes selon la proposition 
\ref{blocindec}, 
et la première implique la troisième selon le lemme \ref{lemdep}.
Il ne reste donc plus qu'à prouver que~la troisième implique l'une 
des deux autres. 
Or il suit du théorème \ref{Prop 3.2.4} que, 
si des représentations $\pi$ et $\pi'$ sont dépendantes, 
alors $\B(\pi')=\B(\pi)$.
\end{proof}

\section{Blocs supercuspidaux} 
\label{sec7}

Soit $\pi$ une $\flb$-représentation supercuspidale de $\G=\GL_m(\D)$.
Notons $\Omega$ sa classe inertielle,~et posons $\B=\B(\pi)$. 
L'objectif de cette section est de prouver le théorème suivant.

\begin{theo}
\label{nerodisepia}
Il existe un corps localement compact non archimédien $\F'$
et une $\F'$-algèbre à division centrale $\D'$ tels que 
les blocs $\Rep_{\flb}(\G,\Omega)$ et $\Rep_{\flb}(\D'^\times,\Omega')$
soient équivalents,
où~$\Omega'$ est la classe inertielle du caractère trivial de 
$\D'^\times$. 
\end{theo}

\subsection{}
\label{dixvarsigma}

Indiquons immédiatement comment déduire du théorème \ref{nerodisepia}
le corollaire suivant.

\begin{coro}
\label{corodisepia}
Notons $\B'$ l'ensemble des représentations irréductibles de $\D'^\times$
dépendantes du caractère trivial.
Alors les blocs $\rep_{\flb}(\G,\B)$ et $\rep_{\flb}(\D'^\times,\B')$
sont équivalents.
\end{coro}

\begin{proof}
L'équivalence de catégories du théorème \ref{nerodisepia}
préserve le fait d'être de lon\-gueur finie.
Elle envoie donc $\rep_{\flb}(\G,\B)$ sur un bloc de 
$\rep_{\flb}(\D'^\times)$ contenu dans $\Rep_{\flb}(\D'^\times,\Omega')$.
Un tel bloc est de la forme $\rep_{\flb}(\D'^\times,\B'')$,
où $\B''$ est égal à
$\B(\chi)$ pour un caractère non ramifié $\chi$ de $\D'^\times$.
Il suffit~alors d'appliquer le foncteur
de torsion par $\chi^{-1}$,
qui induit une équivalence 
entre $\rep_{\flb}(\D'^\times,\B'')$ et $\rep_{\flb}(\D'^\times,\B')$.
\end{proof}

En outre,
la proposition \ref{urgence} montre que,
pour prouver le théorème \ref{nerodisepia},
il suffit de le faire~dans le cas où $\pi$ est de niveau $0$,
ce que nous supposerons désormais. 

Nous allons construire un progénérateur de type fini du bloc 
$\Rep_{\flb}(\G,\Omega)$ et calculer l'algèbre de ses endomor\-phismes. 

\subsection{}
\label{progene}

Fixons un type $(\N,\xi)$ dont l'induite compacte à $\G$ soit isomorphe à 
$\pi$, comme au paragraphe \ref{consnxi3}.
Soit $\xi^0$ la restriction de $\xi$ à $\N^0$,
et soit $\P^0$ l'enveloppe projective de $\xi^0$ dans
$\rep_{\flb}(\N^0)$. 

\begin{lemm}
L'induite compacte~:
\begin{equation}
\label{petitprog}
\ind^\G_{\N^0} (\P^0)
\end{equation}
est projective et de type fini dans $\Rep_{\flb}(\G,\Omega)$.
\end{lemm}

\begin{proof}
Elle est projective en tant qu'induite compacte d'une représentation 
projective,
et de type fini en tant qu'induite compacte d'une représentation de
dimension finie.

Il reste à prouver que cette représentation, qu'on note $\Pi$,
appartient au bloc $\Rep_{\flb}(\G,\Omega)$.~Soit
$\pi'$ un de ses sous-quotients irréductibles.
Celui-ci est de niveau $0$~:~la~re\-présentation de
$\GL_m(\kk_\D)$ sur l'espace des
vecteurs de $\pi'$ invariants par $\N^1$ est non nulle.
Fixons-en une sous-représentation irréductible $\tau$,
ce qui est possible car $\pi'$ est admissible.  
La restriction de $\Pi$ à $\N^0$ se décompose en la somme directe~:
\begin{equation*}
\bigoplus\limits_{g} \ind^{\N^0}_{\N^0\cap\N^{0g}} (\P^{0g})
\end{equation*}
où $g$ décrit les matrices diagonales de $\G$ de la forme
${\rm diag}(\w_\D^{k_1},\dots,\w_\D^{k_m})$
où $k_1,\dots,k_m$ sont~des entiers relatifs tels que $k_1\>\dots\>k_m$.
Il y a donc un $g$ tel que $\tau$ apparaisse comme sous-quotient
irréductible de $\ind^{\N^0}_{\N^0\cap\N^{0g}} (\P^{0g})$.
Pour que cette induite ait des vecteurs non nuls invariants~par $\N^1$,
il faut et suffit que l'induite
$\ind^{\N^0}_{\N^0\cap\N^{0g}} (\xi^{0g})$ en ait également
(rappelons que les sous-quotients de $\P^0$ sont tous isomorphes à $\xi^0$
d'après \cite{Vigb} III.2.9).
La représentation~$\s$ du groupe $\GL_m(\kk_\D)$ dont $\xi^0$
est l'inflation étant cuspidale (elle est même supercuspidale),
ceci n'est possible que si $g$ normalise $\N^0$,
\ie si $k_1=\dots=k_m$.
Supposons que ce soit le cas,
\ie qu'on ait $g=\w_\D^k$ pour un $k\in\ZZ$.
Alors $\tau$ est un sous-quotient
irréductible de $\P^{0g}$, donc $\tau$
est isomor\-phe à $\xi^{0g}$ 
et $\pi'$ est un sous-quotient irréductible de l'induite
compacte $\ind^\G_{\N^0}(\xi^0)$.
Le ré\-sultat voulu suit maintenant de \cite{SEns} Proposition 8.1
(voir aussi la fin du paragraphe \ref{supertypes}).
\end{proof}

Rappelons
(voir par exemple \cite{Pareigis} 4.11)
qu'un objet projectif et de type fini $\Pi$ de $\Rep_{\flb}(\G,\Omega)$ est un
progénérateur si toute représentation irréductible de 
$\Rep_{\flb}(\G,\Omega)$, \ie toute~re\-pré\-sentation irréductible
inertiellement équivalente à $\pi$,
est isomorphe à un quotient de $\Pi$.

\begin{prop}
\label{propetitprog}
La représentation $\ind^\G_{\N^0} (\P^0)$ est
un progénérateur de $\Rep_{\flb}(\G,\Omega)$.
\end{prop}

\begin{proof}
\'Etant donné un $\flb$-caractère non ramifié $\chi$ de $\G$,
il s'agit de prouver que $\pi\chi$ est un quotient de \eqref{petitprog}.
La représentation $\xi^0$ étant un quotient de $\P^0$,
et le foncteur d'induction compacte étant exact,
il suffit de prouver que $\pi\chi$ est un quotient de l'induite compacte
de $\xi^0$ à $\G$,
ce qui suit de ce que $\pi$ est un quotient de ladite induite et
$\chi$ est trivial sur $\N^0$.
\end{proof}

\begin{rema}
\label{remD1}
{Dans le cas particulier où $\pi$ est le caractère trivial de $\G=\D^\times$,
la représentation $\P^0$ est l'induite à $\Oo_\D^\times$ du caractère trivial de 
$\U_\D^{(\ell)}$,
le plus petit sous-groupe ouvert de~$\Oo_\D^\times$ dont~l'indice est une
puissance de $\ell$.
Le progénérateur \eqref{petitprog} de la proposition \ref{propetitprog} est
donc~l'in\-duite com\-pacte~du caractère trivial de $\U_\D^{(\ell)}$ à $\D^\times$.}
\end{rema}

Nous allons maintenant calculer l'algèbre des endomorphismes
du progénérateur $\ind^\G_{\N^0} (\P^0)$.

\subsection{}
\label{datrappel}

Notons $\GB$ le groupe réductif fini $\GL_m(\kk_\D)$
et $\s$ la représentation supercuspidale de $\GB$ dont~$\xi^0$
est l'inflation.
Fi\-xons~une extension $\boldsymbol{t}$ de $\kk_\D$ dans $\Mat_m(\kk_\D)$ 
de degré $m$,
de façon à voir le~groupe multiplicatif
$\TB=\boldsymbol{t}^\times$ comme un tore de $\GB$.
Notons $\vv$ la~valuation $\ell$-adique de $q^n-1$.
La~com\-po\-san\-te $\ell$-primaire $\SB$ de $\TB$
est~donc cy\-cli\-que d'ordre $\ell^\vv$.
Soit $\Sigma$ la $\zlb$-représentation projective de $\GB$ telle que 
la $\flb$-représentation $\Sigma\otimes\flb$ soit 
une enveloppe pro\-jec\-tive de $\s$ dans $\rep_{\flb}(\GB)$.
On a le résultat suivant (Dat \cite{Datltna} Proposition B.1.2).

\begin{prop}
\label{datdl}
Il y a un isomorphisme de $\zlb$-algèbres 
$\End(\Sigma)\simeq\zlb[\SB]$.
\end{prop}

Plus précisément,
une fois fixé un $\flb$-caractère $\Gal(\boldsymbol{t}/\kk_{\F})$-régulier 
$\t$ de $\TB$ correspondant~à~$\s$ par la théorie de Deligne-Lusztig,
la $\zlb$-représentation $\Sigma$ considérée par Dat est 
munie d'une~ac\-tion $\zlb$-linéaire de $\TB$ commutant à celle de $\GB$. 
Il y a donc un homomorphisme~de $\zlb$-algèbres~de~$\zlb[\TB]$
dans $\End(\Sigma)$, induisant par restriction
un isomorphisme~de $\zlb$-algèbres~:
\begin{equation}
\label{isodat}
j : \zlb[\SB] \to \End(\Sigma).
\end{equation}
Par ailleurs, 
il y a une décomposition canonique de $\qlb$-représentations
de $\GB\times\TB$~:
\begin{equation}
\label{decPqlb}
\Sigma \otimes \qlb = \bigoplus\limits_{\a} \V_\a
\end{equation}
indexée sur les $\qlb$-caractères $\a$ de $\SB$,
où $\V_\a$ est isomorphe à $\pi_{\a}\otimes\t\a$,
la représenta\-tion $\pi_{\a}$~étant l'unique
$\qlb$-représenta\-tion cuspidale de $\GB$ relevant $\s$ correspondant
au $\qlb$-caractère $\t\a$
(où $\t$~dé\-si\-gne, par abus de notation,
l'unique relèvement de $\t$ à~$\qlb$ de même ordre que $\t$)
par la théorie~de Deligne-Lusztig.
\'Etendant les scalaires à $\qlb$ dans 
\eqref{isodat},~on~ob\-tient l'isomorphisme canonique de $\qlb$-algèbres~:
\begin{equation*}
\qlb[\SB] \to
\End(\Sigma)\otimes\qlb \simeq \End_{\qlb[\GB]}(\Sigma\otimes\qlb)
\end{equation*}
qui, compte tenu de \eqref{decPqlb}, 
associe à tout élément $f\in\qlb[\SB]$ l'endomorphisme de
$\Sigma\otimes\qlb$~agissant sur
le facteur $\V_\a$ par le scalaire~:
\begin{equation*}
\sum\limits_{x\in\SB} f(x)(\t\a)(x)
=\sum\limits_{x\in\SB} f(x)\a(x)
\end{equation*}
l'égalité provenant de ce que $\t$ est d'ordre premier à $\ell$ et $x$ 
d'ordre divisant $\ell^{\vv}$.

Fixons un~géné\-ra\-teur $\varsigma\in\SB$,
et notons $\tt$ son image dans $\End(\Sigma)$.
On a~: 
\begin{equation}
\label{datdlt}
\End(\Sigma)=\zlb[\tt],
\quad
\tt^{\ell^{\vv}}=1.
\end{equation}
\'Etendant les scalaires à $\qlb$,
l'endomorphisme $\tt$ agit~sur le facteur $\V_\a$ par le scalaire 
$\a(\varsigma) \in \overline{\ZZ}{}_\ell^\times$.

\subsection{}

Notons $\tP^0$ la $\zlb$-représentation~projective de 
$\N$ telle que $\tP^0\otimes\flb$ soit isomorphe à $\P^0$.
\`A~iso\-morphisme près,
c'est~l'in\-flation à $\N^0$ de la~repré\-sen\-tation
$\Sigma$ du paragraphe \ref{datrappel}.
On déduit de~la proposition \ref{datdl} le résultat suivant.

\begin{coro}
\label{datdlcor} 
Il y a un isomorphisme de $\flb$-algèbres 
$\End(\P^0)\simeq\flb[\SB]$. 
\end{coro}

\begin{proof} 
Le résultat suit de la proposition \ref{datdl}
et du fait que la $\flb$-algèbre $\End(\P^0)$~est
isomorphe à $\End(\tP^0)\otimes\flb$.
\end{proof}

La classe~d'iso\-morphisme de $\tP^0$ est~nor\-ma\-lisée par $\N$ car $\xi^0$ 
l'est. 
Rappelons qu'on a fixé une uniformisante $\w_\D$ de $\D$ telle que
$\w_\D^d=\w_\F^{\phantom{d}}$
et que le groupe $\N$~est engendré par $\N^0$ et $\w=\w_\D^b$,
où~$b$ est le cardinal de l'orbite de $\s$ sous $\Gal(\kk_\D/\kk_{\F})$.
Fixons un isomorphisme de~$\zlb$-re\-pré\-sen\-ta\-tions de $\N$~:
\begin{equation}
\label{fixA}
\aa \in \Hom_{\zlb[\N]}(\tP^0,\tP^{0\w})
\end{equation}
ce qui équivaut à fixer un prolongement 
de $\tP^0$ à $\N$ prenant la valeur $\aa$ en $\w$.
\'Etendons les~sca\-lai\-res et restreignons à $\N^0$ dans \eqref{fixA}. 
Compte tenu de \eqref{decPqlb},
pour tout caractère $\a$ de $\SB$,
il~y~a~un unique caractère $\b$ de $\SB$ tel que
$\aa$ envoie $\V_\a$ sur $\V_\b$, 
\ie tel que le~con\-jugué~de $\pi_\a$ par~$\w$ soit isomorphe à $\pi_\b$.
Ceci définit une permutation $\a\mapsto\b$ entre
$\flb$-caractères de $\SB$. 

Rappelons qu'on a défini (voir la définition \ref{hasse})
l'invariant de Hasse $h$ de $\D$,
qui est un entier de $\{1,\dots,d\}$ premier à~$d$.
La conjugaison par $\w$,
qui est égal à $\w_\D^{b}$,
induit donc sur le~corps~rési\-duel~$\kk_\D$ l'automorphisme~:
\begin{equation*}
x \mapsto x^{q^{hb}}.
\end{equation*}
La représentation cuspidale $\pi_\a$ correspondant au caractère
$\t\a$,~sa con\-juguée par $\w$ correspond au caractère~:
\begin{equation*}
(\t\a)^{q^{hb}}
\end{equation*}
qui est conjugué sous $\Gal(\boldsymbol{t}/\kk_\D)$ à $\t\b$ pour un unique
caractère $\b$ de $\SB$, \ie que~: 
\begin{equation*}
(\t\a)^{q^{hb}} = (\t\b)^{q^{di}}
\end{equation*}
pour un $i\in\ZZ$.
Séparant les facteurs d'ordre~pre\-mier à $\ell$ des facteurs d'ordre une 
puissance de~$\ell$,
on obtient~: 
\begin{equation*}
\t^{q^{hb-di}}=\t
\quad\text{et}\quad
\b = \a^{q^{hb-di}}.
\end{equation*}
D'après \cite{MSjl} Paragraphe 3.4,
l'ordre de $q$ modulo l'ordre de $\t$ est égal à $mb$,
et $m$ est premier à l'entier $s=d/b$.
Ainsi $mb$ divise $hb-di$,
\ie que $i$ est solution de l'équation $si\equiv h$ mod $m$.
Il sera commode de choisir $i$ tel que~:
\begin{equation}
\label{defhp}
h' = \frac {h-si} {m} \in \{1,\dots,s\}
\end{equation}
ce~qui détermine $i$ de façon unique.
On observe que l'entier $h'$ ainsi défini est premier à $s$,
car~$h$ est premier à $d$.
On en déduit le résultat suivant.
Soit $\tt$ le générateur de $\End(\tP^0)$ tel que $\tt^{\ell^a}=1$
provenant du générateur $\varsigma\in\SB$ fixé au paragraphe \ref{datrappel} 
( voir \eqref{datdlt}).  

\begin{lemm}
\label{ata}
\begin{enumerate}
\item 
Pour tout caractère $\a$ de $\SB$, on a~:
\begin{equation*}
\b = \a^{q^{hb-di}} =\a^{q^{mbh'}}.
\end{equation*}
\item
On a l'égalité $\aa\tt\aa^{-1}=\tt^{q^{mbh'}}$.
\end{enumerate}
\end{lemm}

\begin{proof}
La première assertion résulte de la discussion qui précède.
Ensuite,
d'après~le pa\-ra\-graphe \ref{datrappel},
l'élément $\tt$ agit sur le facteur $\V_\a$ par le scalaire
$\a(\varsigma)$,~tan\-dis~que $\aa\tt\aa^{-1}$ agit~sur $\V_\a$
comme $\tt$ agit sur
$\V_\b$, \ie par le scalaire $\b(\varsigma)$.
L'endomorphisme $\aa\tt\aa^{-1}-\tt^{q^{mbh'}}$~agis\-sant par $0$ sur
chaque facteur $\V_\a$, il est nul. 
\end{proof}

\begin{rema}
Ceci ne dépend du choix ni de $\t$, ni du générateur $\varsigma\in\SB$,
ni de $\aa$.
\end{rema}

\subsection{}

Notons $\P$ l'induite compacte de $\P^0$ à $\N$.

\begin{theo}
\label{structureE}
La $\flb$-algèbre $\End(\P)$ est
engendrée par deux générateurs $\nn,\uu$
avec les relations~: 
\begin{equation*}
\nn^{\ell^{\vv}}=1, 
\quad
\uu\nn\uu^{-1} = \nn^{q^{mbh'}}.
\end{equation*} 
\end{theo}

\begin{proof}
Notons $\E$, $\E^0$ les algèbres~d'endomor\-phismes de
$\P$, $\P^0$~res\-pec\-tivement.
Par définition de $\P$, on a un morphisme naturel~:
\begin{equation}
\label{plongeonEell}
\E^0 \to \E
\end{equation}
de $\flb$-algèbres.
Pour tout $h\in\N$ et tout vecteur $v$ dans l'espace de $\P^0$,
on note $[h,v]$~l'élément~de $\P$ de support $\N^0 h$ et prenant
la valeur $v$ en $h$.
Ainsi, si $e\in\E^0$,
son image dans $\E$ est~l'endomor\-phisme
$[h,v]\mapsto[h,e(v)]$.
On en déduit que \eqref{plongeonEell} est injective. 
Identifions dorénavant $\E^0$ à son image dans $\E$.
Notons encore $\aa$ l'image de \eqref{fixA} dans
$\Hom_{\flb\N}(\P^0,\P^{0\w})$ par réduction de $\zlb$ à~$\flb$
et définissons 
un~endomor\-phisme de $\flb$-modules $\uu$ de $\P$ par~:
\begin{equation*}
\uu([h,v]) = [\w^{-1} h,\aa(v)].
\end{equation*}
On vérifie que $\uu$ commute à l'action de $\N$, \ie que $\uu\in\E$.
Soit maintenant $\nn\in\E^0$ comme dans le corollaire \ref{datdlcor}. 
Calculons $\uu\nn\uu^{-1}$.
L'appliquant à la fonction $[h,v]$, on trouve~:
\begin{eqnarray*}
\uu\nn\uu^{-1} ([h,v]) &=& \uu\nn([\w^{-1}h,\aa^{-1}(v)]) \\
&=& \uu([\w^{-1}h,\nn\aa^{-1}(v)]) \\
&=& [h,\aa\nn\aa^{-1}(v)]
\end{eqnarray*}
\ie que, d'après le lemme \ref{ata}, on a 
$\uu\nn\uu^{-1}=\nn^{q^{mbh'}}$.

Il ne reste plus qu'à vérifier que $\E$ est engendré par $\E^0$ et $\uu$.
Par réciprocité de Frobenius et décomposition de Mackey,
on a des isomorphismes de $\flb$-espaces vectoriels~:
\begin{equation*}
\E
\simeq \Hom_{\N^0}(\P^0,\P)
\simeq \bigoplus\limits_{i\in\ZZ} \Hom_{\N^0}(\P^0,\P^{0\w^i})
= \bigoplus\limits_{i\in\ZZ} \aa^i\E_0.
\end{equation*}
Ceci se traduit de la façon suivante~: 
étant donné un endomorphisme $f\in\E$,
les applications~: 
\begin{equation*}
f_i : v \mapsto \aa^{-i} f([1,v])(\w^i),
\quad v\in\P^0,
\quad i\in\ZZ,
\end{equation*}
appartiennent à $\E_0$, 
ne sont non nulles que pour un nombre fini de $i$
(car $f([1,v])\in\ind^\N_{\N^0}(\P^0)$ est à sup\-port compact dans $\N$)
et $f$ est égal à la somme des $\uu^{i}f_i$.
\end{proof}

\begin{rema}
\label{remD2}
\begin{enumerate}
\item 
{Le théorème \ref{structureE} généralise des résultats de Dat~\cite{Datltna}~: 
voir la proposition B.1.2(ii.a) dans le cas où $d=1$, 
et la proposition B.2.1(iv) dans le cas où $m=s=1$.}
\item
{Dans le cas où $m=1$, $s=n$, 
par exemple si $\pi$ est le caractère trivial de $\D^\times$,
la remarque \ref{remD1} montre que $\End(\P)$ est l'algèbre de groupe de
$\D^\times/\U_\D^{(\ell)}$,
ce dernier étant isomorphe au~pro\-duit semi-direct de $\SB$,
la composante $\ell$-primaire de $\kk_\D^\times$, 
par $\ZZ$,
un entier $i\in\ZZ$ agissant sur $\SB$
par l'automorphisme $x\mapsto x^{q^{hi}}$.}
\end{enumerate}
\end{rema}

\subsection{}
\label{algdp}

Prouvons maintenant le théorème \ref{nerodisepia}.
Rappelons que $\P$ est l'induite compacte de $\P^0$ à $\N$.

\begin{lemm}
\label{dishoom}
Notons $\Pi$ l'induite compacte de $\P$ à $\G$.
\begin{enumerate}
\item 
La catégorie $\Rep_{\flb}(\G,\Omega)$ est équivalente à la catégorie des
modules à droite sur $\End(\Pi)$.
\item
Le morphisme naturel de $\flb$-algèbres de $\End(\P)$ dans $\End(\Pi)$ 
est un isomorphisme.
\end{enumerate}
\end{lemm}

\begin{proof}
La représentation $\Pi$ étant un progénérateur de $\Rep_{\flb}(\G,\Omega)$
d'après la pro\-position \ref{propetitprog},
il suit par exemple de \cite{Pareigis} 4.11 que le foncteur~:
\begin{equation}
\label{morita}
\V \mapsto \Hom_\G(\Pi,\V)
\end{equation}
est une équivalence de catégories entre $\Rep_\flb(\G,\Omega)$
et la catégorie $\Mod(\End(\Pi))$
des modules~à droi\-te sur $\End(\Pi)$. 
Ensuite, par adjonction, on a un isomorphisme de
$\flb$-espaces vectoriels~:
\begin{equation}
\label{isoFrob}
\End(\Pi) \simeq \bigoplus\limits_{g} \Hom_{\N}(\P,\P^g)
\end{equation}
où $g$ décrit un système de représentants de doubles classes de $\G$
mod $\N$. 
Soit $g\in\G$ tel que~l'es\-pa\-ce $\Hom_{\N\cap\N^g}(\P,\P^g)$ soit non nul.
La restriction de $\P$ à $\N^0$ étant une somme directe de copies~de 
$\P^0$, on trouve que $\Hom_{\N^0\cap\N^{0g}}(\P^0,\P^{0g})$ est non nul, 
donc que $g$ entrelace $\xi ^0$, \ie que~$g\in\N$.
Par conséquent, 
\eqref{isoFrob} donne un~iso\-mor\-phisme de $\flb$-algèbres 
entre $\End(\Pi)$ et $\End(\P)$.
\end{proof}

\begin{rema}
\label{remD3}
Dans le cas particulier où $\pi$ est le caractère trivial de $\D^\times$
(voir les remar\-ques \ref{remD1} et \ref{remD2}),
le foncteur \eqref{morita} est le foncteur des invariants par 
$\U_\D^{(\ell)}$,
qui est distingué dans $\D^\times$.
Par conséquent,
si $\V$ est une représentation dans le bloc principal de $\D^\times$,
\ie dont tous les sous-quotients irréductibles sont des caractères non
ramifiés,
l'espace $\W$ de ses vecteurs invariants par $\U_\D^{(\ell)}$ est stable
par $\D^\times$,
et le quotient $\X=\V/\W$ n'a pas de vecteur $\U_\D^{(\ell)}$-invariant non nul.
Il s'ensuit que $\X$ est nul,
donc que le foncteur \eqref{morita} est le foncteur identité.
\end{rema}

Soit maintenant $\F'$ un corps localement compact non archimédien
dont le corps résiduel soit de cardinal $q'=q^{mb}$,
et soit $\D'$ une $\F'$-algèbre à division centrale de degré réduit $s$
et~d'inva\-riant de Hasse égal à l'entier $h'$ défini par \eqref{defhp},
qui est premier~à~$s$.
Soit $\Pi'$ l'induite compacte à $\D'^\times$ du caractère trivial
de $\U_{\D'}^{(\ell)}$ (remarque \ref{remD1}) 
et soit $\E'$ l'al\-gè\-bre de ses endo\-morphismes.
Comme $q'^s$ est égal à $q^n$,
la valuation $\ell$-adique de $q'^s-1$ est $a$.
D'après la remarque \ref{remD2}, 
l'algèbre 
$\E'$ est engendrée par des~géné\-ra\-teurs $\tt',\uu'$ avec les relations~: 
\begin{equation*}
\tt'^{\ell^{a}}=1,
\quad
\uu' \tt' \uu'^{-1} = \tt'^{q^{mbh'}}.
\end{equation*}
Elle est donc isomorphe à $\E$.
D'après la remarque \ref{remD3}, 
les catégories $\Rep_{\flb}(\D'^\times,\Omega')$ et $\Mod(\E')$ sont identiques.
Le théorème \ref{nerodisepia} s'ensuit.

\begin{rema}
\label{diagcommfinal}
Le diagramme commutatif~:
\begin{equation*}
\xymatrix{\Rep_{\flb}(\G,\Omega) \ar[r] \ar[d] & \Rep_{\flb}(\D'^\times,\Omega')
\ar[d]^{}\\ \Mod(\E) \ar[r] &
\Mod(\E')}
\end{equation*} 
résume la situation~:
les équivalences de catégories verticales sont données par \eqref{morita}
et le foncteur identité $\V'\mapsto\Hom_{\D'^\times}(\Pi',\V')$,
l'équivalence horizontale inférieure est induite par l'isomorphisme
de $\flb$-algèbres entre $\E'$ et $\E$ envoyant $\tt'$ sur $\tt$
et $\uu'$ sur $\uu$, 
et l'équivalence horizontale supérieure
est définie par le fait que le diagramme est commutatif. 
\end{rema}

\begin{rema}
\label{paranormal}
On observe ici un phénomène qui ne se produit pas dans le cas complexe.
{Dans le cas complexe en effet, un bloc supercuspidal de $\GL_m(\D)$
est toujours équivalent 
au bloc de $\F^\times$ contenant le caractère trivial,
\ie qu'on peut choisir $\D'$ égal à $\F'$
(et même $\F'$ égal à $\F$)
dans le théorème \ref{nerodisepia}.
Dans le cas $\ell$-modulaire en revanche,
un bloc de $\rep_\flb(\F^\times)$ contient une seule représentation
irréductible,
tandis que $\rep_\flb(\G,\B(\pi))$
contient tous les $\pi\nu^j$, $j\in\ZZ$.}
\end{rema}

\subsection{}

En guise d'application du théorème \ref{nerodisepia} et du corollaire
\ref{corodisepia}, nous prouvons le résultat suivant.

\begin{prop}
\label{epinards}
Soit $\pi$ une $\flb$-représentation supercuspidale de $\G$.
\begin{enumerate}
\item 
La représentation $\pi$ admet une en\-velop\-pe projective $\Pi$
dans $\rep_\flb(\G)$.
\item
Les sous-quotients irréductibles de $\Pi$ sont tous de la forme 
$\pi\nu^j$ avec $j\in\ZZ$ et,~si $\ell$ ne~di\-vi\-se~pas $q(\pi)-1$,
on peut même supposer que $j=0$.
\end{enumerate} 
\end{prop}

\begin{proof}
Par équivalence, 
il suffit de le prouver dans le cas où $\pi$ est~le~ca\-ractère~trivial de 
$\D^\times$,
ce que nous supposons.
Dans ce cas, $\pi$ contient le type $(\N,\xi)$ où $\xi$ est le caractère
trivial de $\N=\D^\times$.
La première partie de l'énoncé découle donc de la proposition \ref{Propo 3.1.4}.

Soit donc $\P$ l'enveloppe projective du ca\-ractère~trivial de 
$\D^\times$ dans $\rep_\flb(\D^\times)$.
Il nous reste à prouver que
ses sous-quotients irréductibles sont tous de la forme 
$\nu^j$ avec $j\in\ZZ$ et que,~si $\ell$ ne divise pas $q^n-1$,
on peut même supposer que $j=0$.
Soit $\zeta$ un sous-quotient irréductible de $\P$.
D'après le lemme \ref{decdebase5},
il y a une $\qlb$-représentation irréductible $\mu$ de $\D^\times$
dont la réduction mod $\ell$ contienne à la fois $1$ et $\chi$.
Appliquant la proposition \ref{redcusp}
et le corollaire \ref{coromal2q3} à $\mu$,
on trouve le résultat voulu. 
\end{proof}


\section{Le premier espace d'extension dans le cas supercuspidal} 

Soit $\pi$ une $\flb$-représentation supercuspidale de $\G=\GL_m(\D)$. 
Nous déterminons toutes les~re\-présentations irréductibles $\pi'$ de $\G$
telles qu'il existe une extension non scindée de $\pi$ par $\pi'$.

\subsection{}

Dans ce paragraphe,
on cherche les $\flb$-représentations irréductibles de $\D^\times$
ayant une extension non triviale avec le caractère trivial $1$.
D'après le début de la section \ref{sec4},
de telles~repré\-sen\-ta\-tions doi\-vent être des caractères non ramifiés
de $\D^\times$ d'ordre divisant $n$.

Soit donc $\chi$ un~ca\-rac\-tère non ramifié de $\D^\times$
d'ordre divisant $n$.
Si $\chi$ est trivial,
l'espace~d'ex\-ten\-sion $\Ext_{\D^\times}(1,1)$ est toujours non trivial,
car la représentation~:
\begin{equation*}
x \mapsto 
\begin{pmatrix}
1 & \a(x)\\ 0 & 1
\end{pmatrix}
\end{equation*}
où $\a(x)$ est l'image dans $\flb$
de la valuation de la norme réduite de $x$,
est une extension non~tri\-viale~de $1$ par lui-même. 
Supposons donc $\chi$ non trivial.
On pose~:
\begin{equation*}
z=\chi(\w_\D)\in\overline{\mathbb{F}}{}_\ell^\times
\end{equation*}
qui vérifie $z^n=1$ et $z\neq1$.
On cherche à quelle condition l'espace $\Ext^1_{\D^\times}(1,\chi)$
est nul,
\ie à quelle condition toute extension de $1$ par $\chi$ est scindée.

Soit $\M$ une extension de $1$ par $\chi$,
et soit $(e_1,e_2)$ une base de $\M$ sur $\flb$ telle que
$\D^\times$ agisse sur la droite engendrée par $e_1$ par
le caractère $\chi$.
On définit une application
$\a$ de $\D^\times$ dans $\flb$
par~:
\begin{equation*}
x\cdot e_2= e_2 + \a(x)e_1,
\quad
x\in\D^\times.
\end{equation*}
Comme $\M$ est une représentation de $\D^\times$,
cette application $\a$ a la propriété de cocycle~:
\begin{equation}
\label{cocycle}
\a(xy) = \a(x) + \chi(x)\a(y),
\quad
x,y\in\D^\times,
\end{equation}
et l'extension $\M$ est scindée si et seulement s'il y a un
$\l\in\flb$ tel que $\D^\times$ agisse sur $e_2-\l e_1$ par
le caractère trivial,
\ie tel que~:
\begin{equation}
\label{cobord}
\a(x) = \l(\chi(x)-1),
\quad
x\in\D^\times.
\end{equation}
Le caractère $\chi$ étant non ramifié,
\eqref{cocycle} entraîne que 
la restriction de $\a$ à $\Oo^\times_\D$ est un
mor\-phis\-me
de groupes de $\Oo^\times_\D$ dans $\flb$.

\begin{lemm}
\label{pik}
Pour tout $k\in\ZZ$, on a~: 
\begin{equation*} 
\a(\w_{\D}^k) = \a(\w_{\D}^{\phantom{k}}) \cdot \frac {z^k-1} {z-1}. 
\end{equation*}
\end{lemm}

\begin{proof}
La formule est vraie pour $k=0$ et $k=1$ et,
si elle est vraie pour un $k\>1$,
l'identité~:
\begin{equation*}
\a(\w_{\D}^{k+1}) = \a(\w_{\D}^k) + z^k\cdot\a(\w_{\D}^{\phantom{k}})
\end{equation*}
montre qu'elle est vraie pour $k+1$.
Elle est donc vraie pour tout $k\>0$.
De façon analogue,~on vé\-rifie qu'elle est vraie pour $k\<0$. 
\end{proof}

\begin{rema}
On en déduit en particulier que
$\a(\w_{\D}^n)=\a(\w_\F^{\phantom{k}})$ est nul.
\end{rema}

Si $\a$ est nulle sur $\Oo^\times_\D$,
l'extension $\M$ est scindée~:
en effet,
le lemme \ref{pik} prouve qu'on a~la~rela\-tion \eqref{cobord} 
avec $\l=\a(\w_{\D})/(z-1)$.

Supposons maintenant que $\a$ ne soit pas nulle sur $\Oo^\times_\D$.
Comme $\ell\neq p$,
elle est nulle sur~le~pro-$p$-sous-groupe $1+\p_\D$~;
elle induit donc un mor\-phis\-me non nul de groupes
de $\kk^\times_\D$ dans $\flb$, donc $\ell$ divise $q^n-1$.
La condition \eqref{cocycle} entraîne pour tout $x\in\Oo_\D^\times$
les égalités~:
\begin{eqnarray*}
\a(\w_{\D}^{\phantom{1}} x \w_{\D}^{-1}) 
&=& \a(\w_\D^{\phantom{1}}) + z\cdot(\a(x) + \a(\w_\D^{-1})) \\
&=& \a(\w_\D^{\phantom{1}}) + z\cdot\a(x) + z\cdot\a(\w_\D^{\phantom{1}})\cdot
(z^{-1}-1)/(z-1) \\
&=& z \cdot \a(x).
\end{eqnarray*}
Or, 
si $h$ désigne l'invariant de Hasse de $\D^\times$
(voir la définition \ref{hasse}), le~conjugué de $x$ par $\w_\D$
est congru à $x^{q^{h}}$ mod $1+\p_\D$.
On en déduit que,
pour qu'il existe une~ex\-ten\-sion~non scin\-dée de $1$ par $\chi$,
il faut que $\ell$ divise $q^n-1$ et que~:
\begin{equation*}
z = q^{h}.
\end{equation*}
Nous allons prouver que la réciproque est vraie.

\begin{lemm}
Supposons que $\ell$ divise $q^n-1$,
et notons $\chi_h$ le $\flb$-caractère non ramifié de $\D^\times$ 
prenant la valeur $q^h$ en une uniformisante.
Alors $\Ext^1_{\D^\times}(1,\chi_h)$ est non nul. 
\end{lemm}

\begin{proof}
Si le caractère $\chi_h$ est trivial, \ie si $\ell$ divise $q-1$, 
le résultat découle du fait que $\Ext^1_{\D^\times}(1,1)$ n'est
pas trivial, comme nous l'avons vu au début du paragraphe. 
Suppo\-sons maintenant que $\chi_h$ ne soit pas trivial.
D'après ce qui précède, 
il suffit de prouver l'existence d'une application $\a$ de $\D^\times$ 
dans $\flb$ vérifiant \eqref{cocycle} mais pas \eqref{cobord}.
Nous allons construire une~ap\-pli\-cation $\a$
vérifiant~\eqref{cocycle} et non nulle sur $\Oo_\D^\times$~:
elle ne pourra donc pas vérifier \eqref{cobord}.
Comme $\ell$ divise $q^n-1$,
il y a un morphisme de groupes non nul de $\Oo_\D^\times$
dans $\flb$.
Fixons-en un, qu'on~note $\b$.
Tout $x\in\D^\times$ s'écrit de façon unique
$u\w_\D^k$, où $k\in\ZZ$ est sa valuation et où $u\in\Oo_\D^\times$. 
Posons~:
\begin{equation*}
\a(x) = \b(u) + \frac {z^k-1} {z-1},
\quad\text{avec } z=q^h\neq1.
\end{equation*}
Nous allons prouver que $\a$ vérifie \eqref{cocycle}.
Si $y\in\D^\times$, écrivons-le $v\w_\D^l$ avec $l\in\ZZ$ et
$v\in\Oo_\D^\times$.
On a~:
\begin{equation*}
\a(xy) = \b\left(u\w_\D^{k}v\w_\D^{-k}\right) + \frac {z^{k+l}-1} {z-1} 
\end{equation*}
tandis que~:
\begin{eqnarray*}
\a(x) + \chi_h(x)\a(y) &=& \b(u) + \frac {z^k-1} {z-1} + 
z^k\cdot\left(\b(v)+\frac {z^l-1} {z-1} \right).
\end{eqnarray*}
Pour que $\a$ vérifie l'identité \eqref{cocycle},
il faut et suffit donc que~:
\begin{equation*}
\b\left(\w_\D^{\phantom{1}} v\w_\D^{-1}\right) = z\cdot\b(v)
\end{equation*}
pour tout $v\in\Oo_\D^\times$,
ce qui découle immédiatement de ce que,
comme remarqué plus haut,
le~conju\-gué de $v$ par $\w_\D$
est congru à $v^{q^{h}}$ mod $1+\p_\D$.
\end{proof}

En résumé, on a le résultat suivant.

\begin{prop}
\label{chih}
L'ensemble des $\flb$-caractères $\chi$ de $\D^\times$ tels que 
$\Ext^1_{\D^\times}(1,\chi)$ soit non nul est~:
\begin{enumerate}
\item 
réduit au caractère trivial si $\ell$ ne divise pas $q^n-1$, 
\item
formé du caractère trivial et du caractère non ramifié
$\chi_h:\w_\D\mapsto q^{h}$ si $\ell$ divise $q^n-1$.
\end{enumerate}
\end{prop}

\subsection{}

Dans ce paragraphe,
nous généralisons la proposition \ref{chih}
au cas d'une représentation~super\-cus\-pi\-dale quelconque $\pi$
de $\G=\GL_m(\D)$.
Nous allons utiliser la description de $\pi$ par la théorie des types simples
donnée au paragraphe \ref{flapflip},
dont nous utilisons les notations.

Notons $h(\pi)$ l'invariant de Hasse de l'algèbre à division $\C$
apparaissant dans \eqref{melon}.
Si $h$ est l'invariant de Hasse de $\D$~et 
si $g$ est le degré de $\E$ sur $\F$ (avec les notations du paragraphe
\ref{flapflip}),
alors $h(\pi)$ est le reste dans~la division euclidienne de
$gh/(g,d)$ par $c=d/(g,d)$.
C'est un entier premier à $c$.

\begin{prop}
\label{pih}
Soit $\pi$ une $\flb$-représentation supercuspidale de $\G$. 
L'ensemble des $\flb$-re\-pré\-sentations supercuspidales $\pi'$ de $\G$ telles
que $\Ext^1_{\G}(\pi,\pi')$ soit non nul est~:
\begin{enumerate}
\item 
réduit à $\pi$ si $\ell$ ne divise pas $q(\pi)-1$, 
\item
formé de $\pi$ et de la représentation 
$\pi\nu^{-h(\pi)}$ si $\ell$ divise $q(\pi)-1$.
\end{enumerate}
\end{prop}

\begin{proof}
Dans le cas où $\ell$ ne divise pas $q(\pi)-1$, 
le résultat découle de la propo\-sition \ref{ext1cuspniveaunonnul}.
Supposons désormais que $\ell$ divise $q(\pi)-1$.
Selon la proposition \ref{ext1cuspniveaunonnul}, 
on peut supposer que $\pi'$ est isomorphe à $\pi\nu^j$ pour un entier
$j\in\ZZ$.

Supposons d'abord que le résultat soit vrai pour les représentations 
supercuspidales de niveau $0$~:
on va en déduire le~résultat dans le cas de niveau non nul
en raisonnant comme dans~la preuve de la proposition \ref{ext1cuspniveaunonnul}. 
Appliquons le foncteur $\GG$ introduit au paragraphe \ref{foncGG},
envoyant $\pi$ sur la représentation supercuspidale $\pi_0$ de niveau $0$
de $\G_0$.
On trouve que~:
\begin{equation*}
\Ext^1_{\G}(\pi,\pi') \neq \left \{ 0 \right \}
\quad\Leftrightarrow\quad
\Ext^1_{\G_0}(\GG(\pi),\GG(\pi')) \neq \left \{ 0 \right \}.
\end{equation*}
D'après le cas de niveau $0$,
et comme $q(\pi_0)=q(\pi)$ d'après le lemme \ref{quicklemma}, 
la représentation $\GG(\pi')$ est isomorphe à $\pi_0$ ou à
$\pi_0^{\phantom{1}}\nu_0^{-h(\pi)}$, 
où $\nu_0$ est le caractère non ramifié
``valeur absolue de la norme réduite'' de $\G_0$.
Par ailleurs,
d'après le lemme \ref{adieulescons},
et comme la restriction de $\nu$ à $\G_0$ est égale à $\nu_0$,
la représentation $\GG(\pi\nu^{-h(\pi)})$ est isomorphe à
$\pi_0^{\phantom{1}}\nu_0^{-h(\pi)}$.
Le foncteur $\GG$ étant une équivalence de catégories, on trouve le résultat 
annoncé. 

Supposons maintenant que $\pi$ soit de niveau $0$.
On a donc $q(\pi)=q^n$ et $h(\pi)=h$.
Reprenons les notations du paragraphe \ref{algdp},
et notamment celles de la remarque \ref{diagcommfinal}.
On a une équivalence de catégories entre
$\Rep_\flb(\G,\Omega)$ et $\Rep_\flb(\D'^\times,\Omega')$,
que l'on note $\GGG$,
et un isomorphisme d'algèbres entre $\E$ et $\E'$.
Nous allons utiliser le même argument que ci-dessus~:
il suffit pour cela de~vé\-rifier que 
$\GGG(\pi\nu^j)$ est isomorphe à $\GGG(\pi)\nu'^j$,
où $\nu'$ est la caractère 
``valeur absolue~de la norme réduite'' 
de $\D'^\times$.
Pour cela,
nous allons suivre le raisonnement de la preuve du lemme \ref{adieulescons}.

D'abord,
par analogie avec le lemme \ref{Lemme 4.1.4},
on a le résultat suivant~:
si l'on note $\M$ le $\E$-module à droite $\Hom_\G(\Pi,\pi)$, 
alors $\Hom_\G(\Pi,\pi\nu^j)$ est le $\E$-module obtenu en tordant 
$\M$ par~le ca\-ractère de $\E$ défini par
$\tt\mapsto1$ et $\uu\mapsto\nu^j(\w)$.

Ensuite, 
il existe un unique caractère non ramifié $\mu'$ de~$\D'^\times$ tel que 
$\GGG(\pi\nu^j)$ soit égal à $\GGG(\pi)\mu'$.
Considéré comme un module sur $\E'$,
il est 
obtenu en tordant $\GGG(\pi)$
par~le ca\-ractère de $\E'$ 
défini par $\tt'\mapsto1$ et $\uu'\mapsto\mu'(\w')$,
où $\w'$ est une uniformisante de $\D'$,
ce qui entraîne $\mu'(\w') = \nu^j(\w)$. 
Or $\nu(\w)$ est égal à $q^{-mb} = q'^{-1}$. 
On en déduit que $\mu'$ est égal à $\nu'^j$. 
\end{proof}


\appendix

\section{}
\label{App}

Comme nous l'avons mentionné dans l'introduction,
la preuve de \cite{SEns} Proposition 1.3 est~in\-correcte.
(Dans la preuve de \cite{SEns} Lemme 1.4,
un sous-quotient $\l$ de $\xi$ peut être~cus\-pi\-dal,~ce~qui
em\-pê\-che d'appliquer l'hypothèse de récurrence.)
La preuve du théorème de décomposition \cite{SEns} Theorem~10.4
reposant sur cette proposition,
nous devons expliquer pourquoi ce théorème reste valable. 
Plus précisément, 
dans l'article \cite{SEns},
le seul résultat de la section 1 utilisé dans la preuve du théorème 10.4 est 
le~co\-rollaire 1.15.
Nous allons prouver que ce corollaire reste valable.

Dans tout cet appendice, 
$\GB$ est un groupe linéaire général sur un corps fini de caractéristique $p$
et $\R$ est un corps algébriquement clos de caractéristique différente de $p$. 

\begin{prop}
\label{sauvetagefini}
Soient $\pi$ et $\pi'$ des représentations irréductibles de $\GB$.
On suppose qu'il~y a un entier $i\>0$ tel que $\Ext_\GB^i(\pi',\pi)$
soit non nul.
Alors $\pi$ et $\pi'$ ont même support supercuspidal. 
\end{prop}

\begin{proof}
Soit $\P$ l'enveloppe projective de $\pi$.
Nous allons prouver que tout sous-quo\-tient irré\-duc\-tible
$\a$ de $\P$ a le même support
supercuspidal que $\pi$.
Si $\Q$ est l'enveloppe projective de $\a$,
on a un morphisme non nul de $\Q$ dans $\P$. 
En relevant $\P$, $\Q$ à $\zlb$ et en étendant~les~scalai\-res~à~$\qlb$,
on en déduit
(comme au paragraphe \ref{par46} ci-dessus)
qu'il y a une $\qlb$-représentation~ir\-ré\-ductible $\d$ dont la
réduction~mod $\ell$ contienne $\pi$ et $\a$.
Les représentations $\pi$ et $\a$ ont donc le même support supercuspidal,
qui est le support supercuspidal mod $\ell$ de $\d$. 

Prouvons maintenant que $\pi$ et $\pi'$ ont même~sup\-port supercuspidal. 
Le cas où $i=0$ est~immé\-diat.
Supposons donc que $i\>1$.
Soit $\Q$ l'enveloppe projective de $\pi'$ et soit
$\Q_1$ la sous-représen\-tation propre maximale de $\Q$.
Appliquant le foncteur $\Hom_{\GB}(-,\pi)$,
et $\Q$ étant projectif, on~a~:
\begin{equation*}
\Ext_\GB^i(\pi',\pi) \simeq \Ext_\GB^{i-1}(\Q_1,\pi).
\end{equation*}
Il existe donc un sous-quotient irréductible $\a$ de $\Q_1$ tel que
$\Ext_\GB^{i-1}(\a,\pi)$ soit non nul.
Par récur\-ren\-ce sur $i$,
les représentations $\pi$ et $\a$ ont le même support supercuspidal.
Par ailleurs, $\a$ a le même support supercuspidal que $\pi'$ car il
est un sous-quotient de $\Q$.
\end{proof}

On en déduit (par exemple en utilisant le lemme \ref{vanessa})
que les sous-quotients irréductibles d'une
représentation indécomposable de longueur finie de $\GB$
ont tous le même support supercuspidal.
Par conséquent,
les conclusions de \cite{SEns} Proposition 1.7, Lemme 1.11
sont vraies pour le groupe $\GB$. 
Il~s'ensuit que les théorèmes de décomposition
\cite{SEns} Théorème 1.12, Corollaire 1.15
sont vrais pour $\GB$.

\providecommand{\bysame}{\leavevmode ---\ }

\end{document}